\documentclass{article}
\usepackage{geometry}
\usepackage[utf8]{inputenc}
\usepackage{graphicx}
\usepackage{mathtools}
\usepackage{pdflscape}
\usepackage{listings}
\usepackage[pdflang={en-US}]{hyperref}
\usepackage{caption}
\usepackage{bbm}
\usepackage{amsthm}
\usepackage{amsmath}
\usepackage{color}
\usepackage{amssymb}
\usepackage{authblk}
\usepackage{natbib}
\usepackage{subfig}

\usepackage{algorithm}
\usepackage{algorithmic}

\usepackage{booktabs}       
\usepackage{amsfonts}       
\usepackage{nicefrac}       
\usepackage{microtype}      
\usepackage{xcolor}         
\usepackage{times}

\newtheorem{lemma}{Lemma}
\newtheorem{theorem}{Theorem}
\newtheorem*{theorem*}{Theorem}

\newcommand{\1}{\mathbbm{1}}
\newcommand{\E}{\mathbb{E}}

\newcommand{\cP}{\mathcal{P}}

\newcommand{\cQ}{\mathcal{Q}}
\newcommand{\cC}{\mathcal{C}}
\newcommand{\cS}{\mathcal{S}}
\newcommand{\cW}{\mathcal{W}}
\newcommand{\cZ}{\mathcal{Z}}

\newcommand{\nn}{\nonumber}

\title{As Easy as ABC: Adaptive Binning Coincidence Test\\ for Uniformity Testing}
\author{Sudeep Salgia}
\author{Qing Zhao}
\author{Lang Tong}
\affil{School of Electrical \& Computer Engineering, Cornell University, Ithaca, NY, \emph{\{ss3827,qz16,lt35\}@cornell.edu} }

\date{}

\begin{document}

\maketitle

\begin{abstract}%
  We consider the problem of uniformity testing of Lipschitz continuous distributions with bounded support. The alternative hypothesis is a composite set of Lipschitz continuous distributions that are at least $\varepsilon$ away in $\ell_1$ distance from the uniform distribution.  We propose a sequential test that adapts to the unknown distribution under the alternative hypothesis. Referred to as the Adaptive Binning Coincidence (ABC) test, the proposed strategy adapts in two ways. First, it partitions the set of alternative distributions into layers based on their distances to the uniform distribution. It then sequentially eliminates the alternative distributions layer by layer in decreasing distance to the uniform, and subsequently takes advantage of favorable situations of a distant alternative by exiting early. Second, it adapts, across layers of the alternative distributions, the resolution level of the discretization for computing the coincidence statistic. The farther away the layer is from the uniform, the coarser the discretization is needed for eliminating/exiting this layer. It thus exits both early in the detection process and quickly by using a lower resolution to take advantage of favorable alternative distributions. The ABC test builds on a novel sequential coincidence test for discrete distributions, which is of independent interest. We establish the sample complexity of the proposed tests as well as a lower bound.
\end{abstract}

\vspace{1em}


\section{Introduction}
Consider the following composite hypothesis testing problem: Given  samples of a random variable with density function $f$, one is to determine whether $f$ is the uniform distribution $u$ over $[0,1]$ (the null hypothesis) or belongs to  $\cC(\varepsilon)$ which consists of distributions   $\varepsilon$  distance  away (in $\ell_1$) from the uniform distribution. The objective  is to minimize the sample complexity 
subject to the constraint of both the Type I and Type II errors being capped below a specified threshold $\delta \in (0, 1)$. \\

It turns out that the above hypothesis testing problem is not testable~\citep{Adamaszek2010}; no algorithm can achieve diminishing error probability as the number of samples increases.  Some restrictions on the set $\cC(\varepsilon)$ of distributions are therefore necessary. One class of distributions that are testable is the class of monotone distributions \citep{Adamaszek2010, Acharya2013}, for which a sample complexity of the order $O(1/\varepsilon^2)$ has been established.
To our best knowledge, general conditions on testability are unknown, and there are no results on uniformity testing of continuous distributions beyond the monotone case. \\

There is, however, an extensive literature on uniformity testing of discrete distributions with a support size of $m$. The problem dates back to the so-called empty-box problem first posed by~\cite{David1950} and later generalized by~\cite{Viktorova1966}. David cast the problem as throwing balls in $m$ boxes and proposed to use the number of empty boxes or the number of boxes containing exactly one ball (a.k.a the coincidence number)  as  test statistics for uniformity testing\footnote{The work by~\cite{David1950} considered the problem of testing whether samples are drawn i.i.d. from a known continuous distribution. The problem was reduced to a discrete problem by quantizing samples into $m$ bins without discussing the choice of $m$. The heuristic approaches he proposed are for solving the discrete problem.}. \\

While earlier results address the asymptotic properties of uniformity tests, contemporary perspectives focus more on the  finite sample properties in the form of sample and computation complexities. See \cite{Goldreich2017} and references therein.   A particularly interesting solution  was given by \cite{Paninski2008} where he proposes the use of `coincidence' test statistic---the number of one-ball boxes---and established a sample complexity in the form of $O(\sqrt{m}/\varepsilon^2)$ (under the restriction that $\varepsilon = \Omega(m^{-1/4})$).  Paninski also shows that the sample complexity of the coincidence test is  order optimal by providing a matching lower bound. Several other algorithms have also been proposed over the years that use different test statistics, e.g., the $\chi^2$-statistic employed in \cite{Acharya2015}. \\

All existing algorithms for uniformity testing are  batch algorithms in the sense that all samples are collected prior to performing the test, which makes it necessary to  focus on the most challenging alternative distributions (i.e., those that are at the minimum distance $\varepsilon$ away from the uniform).  While such approaches are sufficient to obtain minimax-optimal sample complexity, they would result in significantly suboptimal sample complexity for almost all instances in the class of alternate distributions.  Such suboptimality may have severe consequences in practice.  For example, when the alternative hypotheses represent anomalies in critical infrastructure, it is often more crucial to detect quickly those severe anomaly distributions far away from the normal state\footnote{In most applications, the probabilistic model of the normal state is known, which can be transformed to a uniform distribution. Anomaly detection can then be cast as uniformity testing.}, as more severe anomalies often carry more risk.  It is therefore highly desirable to have a test adaptive to the underlying alternative distribution, which motivates the sequential uniformity test considered in this work.

\subsection{Main results}

We consider uniformity testing of Lipschitz continuous (density) distributions, which is arguably more general than the class of monotone distributions studied in~\cite{Adamaszek2010} and~\cite{Acharya2013}. To our best knowledge, this problem has not been studied. 
Another theme that separates this work from existing literature on coincidence-based uniformity testing is the sequential aspect of the proposed tests that adapts to the underlying alternative distribution.  \\

Referred to as the Adaptive Binning Coincidence (ABC) test, the proposed strategy adapts to the unknown alternative distribution in two ways. First, it partitions the set of alternative distributions into layers based on their distances to the uniform distribution. It then sequentially eliminates the alternative distributions layer by layer in decreasing distance to the uniform, and subsequently takes advantage of favorable situations of a distant alternative by exiting early. Second, it adapts, across layers of the alternative distributions, the resolution level of the discretization for computing the coincidence statistic. The farther away the layer is from the uniform, the coarser the discretization is needed for eliminating/exiting this layer. It thus exits both \emph{early} in the detection process and \emph{quickly} by using a lower resolution to take advantage of favorable alternative distributions. We establish the sample complexity of the ABC test as well as a lower bound for uniformity testing of Lipschitz continuous distributions. While the sample complexity of ABC does not meet the lower bound, its adaptivity is evident from the sample complexity analysis. This is also the first test that demonstrates the testability of Lipschitz continuous distributions. Whether coincidence statistic is sufficient to achieve the optimal sample complexity is an open question for further investigation.  \\

The ABC test builds on a novel sequential coincidence test for discrete distributions, which is of independent interest. Due to the adaptivity, this sequential test  improves the sample complexity under the alternative hypothesis from $O(\sqrt{m}/\varepsilon^2)$ of Paninski's batch algorithm to $O(\gamma^{-2}\sqrt{m} \log(1/\gamma))$,
where  $\gamma$ is the  distance of the underlying alternative distribution  to the uniform and is greater than the minimum distance $\varepsilon$. This demonstrates that the sequential coincidence test adapts to the \emph{realized} distance $\gamma$ in an optimal order (up to a logarithmic factor) in terms of sample complexity. \\

\subsection{Related Work}

The problem of studying, estimating and  testing properties of an unknown distribution has an extensive literature in statistics and theoretical computer science communities. The class of problems includes problems like estimating properties like entropy, testing for structure of distributions, like monotonicity and independence~\citep{Batu2004, Canonne2015}, testing for closeness to a given distribution which includes as a special case of testing for identity and uniformity~\citep{Chan2013, Goldreich2011}. In this work, since we only deal with the problem of uniformity testing, we focus on the works relevant to this particular problem and refer the interested reader to the excellent survey by \cite{Canonne2008} which extensively covers all the topics. \\

The problem of uniformity testing is among the most widely studied problem among this class and has largely been studied for the case of discrete distributions. An interesting and optimal solution to the problem was given by~\cite{Paninski2008}. Paninski developed a test that employs the test statistic called as $K_1$, which counts the number of symbols that have been observed exactly once in the sample set, similar to the one-ball boxes approach suggested by~\cite{David1950}. Paninski proved a sample complexity of $O(\varepsilon^{-2} \sqrt{m})$ for the regime of $\varepsilon = \Omega(m^{-1/4})$ for his proposed approach and also provided a matching lower bound thereby establishing the optimality of the proposed method. Over the years several other estimators have also been proposed based on $\ell_2$ distance \citep{Batu2001, Batu2013, Goldreich2011, Chan2013} and modified $\chi^2$-test statistic \citep{Acharya2015}. Several of them guarantee order-optimal sample complexity for all values of $\varepsilon$ in contrast to the coincidence test proposed by Paninski which operates only in the sparse regime. \\

A common thread among these approaches is that they are batch algorithms. They are designed to collect all the samples together in a single batch and declare the result based on the samples in that batch. This necessitates the batch size to be large enough and tuned to the most difficult distributions in the alternative hypothesis class. In contrast to this fixed-sample-size approach, our proposed algorithm is sequential in nature. It sequentially takes samples in mini-batches and and is designed to terminate at an instant adapted to the \emph{realized} unknown distribution instead of the worst-case scenario, resulting in improved sample complexities. Furthermore, in \cite{Paninski2008} the problem was only considered in the asymptotic regime whereas we extend the scope to finite-time regime with additional analysis particularly to deal with the sequential nature of the test. \\

The literature on uniformity testing of continuous distributions is slim. As mentioned previously, it is necessary to impose a certain structure on the class of distributions being considered in order to ensure feasibility of the problem. For the case when the underlying distribution is monotone, \cite{Adamaszek2010} and \cite{Acharya2013} have proposed algorithms for the problem of uniformity testing that offer optimal sample complexities. In this work, we consider a different structure and study the problem on the class of Lipschitz continuous distributions with bounded support.

\section{Uniformity Testing of Discrete Distributions}
\label{sec:discrete}

In this section, we consider uniformity testing of discrete distributions and develop a sequential test employing the coincidence statistic. The results obtained for the discrete problem form the foundation for tackling the continuous problem in the next section.  \\

The key property of this sequential coincidence test (SCT) is that it adapts to the unknown alternative distribution in the composite set. More specifically, the sample complexity of SCT under the alternative hypothesis scales optimally with respect to the distance $\gamma$ of the realized alternative distribution to the uniform. This is  in sharp contrast to  batch algorithms whose sample complexity is determined by the worst-case alternative distribution seeing the minimum distance $\varepsilon$ to the uniform. 

\subsection{Problem Formulation}
\label{sub:discrete_prob_des}

Consider a binary hypothesis testing problem where the null hypothesis $H_0$ is the uniform distribution $u$ with a support size of $m$. Without loss of generality, we assume that the support set is $\{1,2, \dots, m\}$ denoted by $[m]$. The alternative hypothesis $H_1$ is composite: it consists of all distributions over $[m]$ whose $\ell_1$ distance to $u$ is no smaller than $\varepsilon$. More specifically, let  $\cC(\varepsilon)$ denote the composite set of alternative distributions under $H_1$, we have
\begin{align*}
    \cC(\varepsilon) = \{ q \in \cP([m]): \| q - u \|_1 > \varepsilon \},
\end{align*}
where $\cP([m])$ denotes all distributions over $[m]$, $\|q - u\|_1 = \sum_{i =1}^m |q_i - 1/m|$ is the $\ell_1$ distance between distribution $q$ and the uniform distribution $u$.  \\

For the hypothesis testing problems, i.i.d. samples are drawn from either $u$ (if $H_0$ is true) or a specific distribution in $\cC(\varepsilon)$ (if $H_1$ is true), unknown to the decision maker. The goal is to determine, based on the random samples, which hypothesis is true. The probabilities of false alarm and miss detection need to be capped below a given $\delta$ ($\delta \in (0, 1)$) for all alternative distributions in $\cC(\varepsilon)$.

\subsection{Sequential Coincidence Test}
\label{sub:sct_des}

Existing work on uniformity testing all focuses on batch methods (a.k.a. the fixed-sample-size tests). Specifically, based on the reliability constraint $\delta$ and the minimum separation $\varepsilon$ between $H_0$ and $H_1$, the number of required samples is pre-determined to ensure $\delta$-reliability in the worst case of a closest (i.e., distance $\varepsilon$) alternative.  \\

We propose a sequential test SCT that adapts to the unknown alternative. When the alternative is at a distance greater than $\varepsilon$ from $u$, the sequential test takes advantage of the favorable situation and exits towards $H_1$ with fewer samples. In particular, the sample complexity of SCT scales optimally with the \emph{realized} distance $\gamma$ rather than the minimum distance $\varepsilon$.   \\

SCT employs the coincidence statistic. For a give set of samples $\cS$, the coincidence $K_1(\cS) $ is the number of symbols in $[m]$ that appear exactly once in $\cS$. Specifically, let $n_j$ denote the number of appearances of symbol $j$ in $\cS$. Then 
\begin{align*}
    K_1(\cS) = \sum_{j = 1}^m \1 \{ n_j = 1 \},
\end{align*}
where $\1\{ \cdot \}$ is the indicator function. Let $K_1(n)$ denote the coincidence number of $n$ i.i.d. samples drawn from a given distribution~$p$. It is a random variable whose distribution is determined by $n$ and~$p$. We then introduce the constant $c_u(n)$, which is the expected value of $K_1(n)$ under the uniform distribution:
\begin{align*}
    c_u(n) = \E_u[K_1(\cS)],~~~\mbox{where } |\cS| = n,~\cS\stackrel{\mbox{\footnotesize i.i.d.}}{\sim} u.
\end{align*}

SCT proceeds in epochs. In each epoch the test takes $\Theta(\sqrt{m})$ additional samples. At the end of each epoch, based on all the samples $\cS$ collected so far,  the algorithm decides whether there is sufficient evidence to exit towards $H_1$. This decision is made by comparing the difference between $K_1(\cS)$ and $c_u(|\cS|)$ to a carefully chosen threshold. If the difference exceeds the threshold (indicating a sufficient separation between the coincidence number of the samples $\cS$ and the expected coincidence number of the uniform), the algorithm terminates and declares $H_1$. Otherwise, the algorithm enters the next epoch. In the event that the process reaches the maximum number $\Theta(\varepsilon^{-2})$ of epochs without exiting towards $H_1$, the algorithm terminates and declares $H_0$.  A pseudo code for the algorithm is given in Algorithm~\ref{alg:SCT}. \\

The sequential detection process of SCT can be visualized as peeling an onion: the core of the onion is the uniform distribution and its $\varepsilon$-neighbors; the layers represent alternative distributions at increasing distance to the uniform distribution\footnote{The epoch structure of SCT effectively quantizes the distance to $u$, hence forming a finite partition of the alternative distributions in $\cC(\varepsilon)$. More specifically, $\cC(\varepsilon)$ is partitioned into $\kappa$ layers, where $\kappa= 112\varepsilon^{-2}$ is the maximum epoch number defined in Algorithm~\ref{alg:SCT}. Each epoch peels off one layer.}. Each epoch peels one layer of the onion, either by exiting towards $H_1$ (if the realized alternative distribution resides in this layer) or by eliminating this set of alternative distributions and moving to the next layer closer to the core. If all outer layers are eliminated (i.e., all alternative distributions in $\cC(\varepsilon)$ are eliminated), SCT terminates and declares $H_0$. The ability of peeling the onion layer by layer is rooted in the fact that when the samples $\cS$ are drawn from a distribution $\gamma$-distance away from $u$, the difference in coincidence numbers $c_u(n_k) - K_1(\cS)$  scales proportionally with $\gamma^2$ . This difference hence exceeds the threshold early when $\gamma$ is large (i.e., when the alternative distribution resides farther from the core of the onion).

\subsection{Sample Complexity}
\label{sub:sct_samp_comp}

The expected sample complexity of SCT is characterized in the following theorem.

\begin{theorem}
For $m > m_0$ and $\varepsilon = \Omega(m^{-1/8})$, where $m_0 = \min \{ l: 1123 l^{1/4} e^{-0.25\sqrt{l}} \leq \delta \varepsilon^2 \}$, we have
\begin{itemize}
    \item Under $H_1$ with an alternative distribution $p$ that is $\gamma$ away from $u$, the expected sample complexity of SCT is $\displaystyle O\left(\frac{\sqrt{m}}{\gamma^2} \sqrt{ \log\left( \frac{1}{\gamma} + \frac{1}{\delta} \right)}\right) $.
    \item Under $H_0$, the expected sample complexity of SCT is $\displaystyle O\left(\frac{\sqrt{m}}{\varepsilon^2}  \sqrt{\log\left( \frac{1}{\varepsilon} + \frac{1}{\delta} \right)}\right) $.
    \item Under both $H_1$ and $H_0$, the probability of correct detection under SCT is at least $1 - \delta$. \label{thm:sct_samp_complexity}
\end{itemize}
\end{theorem}

\begin{proof}
See Appendix~\ref{proof:theorem_sct_samp_complexity}.
\end{proof}

As evident from the above theorem, the sample complexity of SCT under $H_1$ adapts to the distance $\gamma$ of the alternative distribution $p$ to the uniform distribution. Since $p \in \cC(\varepsilon)$, we have $\gamma > \varepsilon$, implying that the sample complexity is smaller than the fixed-sample-size approaches which offer a sample complexity of $O(\varepsilon^{-2}\sqrt{m})$.  Moreover, the adaptivity of SCT to the realized distance $\gamma$ is order-optimal (up to a logarithmic factor). This can be shown by noting that even with the knowledge of $\gamma$, the lower bound given by Paninski dictates  $\Omega(\gamma^{-2}\sqrt{m})$ samples to ensure a constant probability of reliability.  \\

\begin{algorithm}
	\caption{Sequential Coincidence Test (SCT)}
	\label{alg:SCT}
	\begin{algorithmic}
		\STATE {\bfseries Input:} $m, \varepsilon$,  $\delta \in (0,1)$ 
		\STATE Set $k \leftarrow 1$, $t \leftarrow 0$, $\kappa = 112\varepsilon^{-2}$, $\cS = \{\}$
	    \WHILE{$k \leq  \kappa$ }
	    \STATE $n_k \leftarrow \left\lceil k \sqrt{m \log(k+2/\delta)} \right\rceil$
	    \STATE $\displaystyle \tau_k \leftarrow 7n_k\sqrt{\frac{\log(k+2/\delta)}{m}}$
	    \REPEAT
	    \STATE Obtain a sample $X_t$
	    \STATE $\cS \leftarrow \cS \cup X_t$
	    \STATE $t \leftarrow t + 1$
	    \UNTIL{$t == n_k$}
	    \IF{$Z_k := c_u(n_k) - K_1(\cS) > \tau_k$}
	    \STATE Output $\leftarrow H_1$
	    \STATE \textbf{break}
	    \ENDIF
	    \STATE $k \leftarrow k + 1$
	    \ENDWHILE
	    \IF{$k > \kappa$}
	    \STATE Output $\leftarrow H_0$
	    \ENDIF
	    \RETURN Output
	\end{algorithmic}
\end{algorithm}

Furthermore, in addition to the near-optimal scaling with $\gamma$, SCT offers better scaling of sample complexity with respect to $\delta$, the error probability, as compared to the batch algorithms. In particular, the batch algorithms are designed to guarantee a certain constant probability of error, and the common technique to extend such tests to guarantee an arbitrary probability of error $\delta$ is to repeat the test sufficiently many times so that the result declared by the majority vote meets the confidence requirements. Such an approach results in a $\log(1/\delta)$ dependence of sample complexity on $\delta$ as opposed to the $\sqrt{\log(1/\delta)}$ offered by SCT. Thus, the refined analysis required to analyze the sequential coincidence test not only demonstrates  adaptivity to the underlying distribution but also results in improved dependence on $\delta$.  \\

In addition to the dependence on $\gamma$ and $\delta$, we would also like to briefly mention the regime of input parameters $m$ and $\varepsilon$ for which this result holds. The lower bound $m_0$ is required to ensure the support size is large enough to ensure a confidence of $\delta$ in the sparse regime. Moreover, the regime of $\varepsilon$ for which this result is applicable can also in part be attributed to the fact that the coincidence statistic works well only in the sparse regime. We believe that the $\varepsilon = \Omega(m^{-1/8})$ requirement as opposed to the standard requirement of $\varepsilon = \Omega(m^{-1/4})$ for sparse regime is merely an analysis artifact and can be improved using better techniques for bounding the moment generating function of $K_1$.

\section{Uniformity Testing of Continuous Distributions}

\subsection{Problem Formulation}
We now consider uniformity testing of continuous distributions. As shown by~\cite{Adamaszek2010}, the problem is ill posed unless certain regularity conditions are imposed on the set of continuous distributions. In this work, we consider Lipschitz continuous distributions with bounded support.
Specifically, the null hypothesis $H_0$ is the uniform distribution $u$  over $[0,1]$. The alternative composite hypothesis $H_1$ is the set of $L$-Lipschitz distributions whose $\ell_1$ distance to $u$ is no smaller than $\varepsilon$. Let $\cP([0,1];L)$ denote the set of distributions over $[0,1]$ that are absolutely continuous with the Lebesgue measure on $[0,1]$ and whose density functions are $L$-Lipschitz. Specifically, for all distributions $q \in \cP([0,1];L)$, the density functions  $f_q(x)$  satisfies, for all $x, y \in [0,1]$,
\begin{align*}
    |f_q(x) - f_q(y)| \leq L |x - y|.
\end{align*}
The composite set of alternative distributions under $H_1$ is given by
\begin{align*}
    \cC_L(\varepsilon) = \left\{ q \in \cP([0,1]; L) :  \| q - u \|_1 = \int_0^1 |f_q(x) - 1| \ \mathrm{d}x > \varepsilon \right\}.
\end{align*}
The objective of the uniformity testing is the same as in the discrete problem: to determine, with a reliability constraint of $\delta$, whether the observed random samples are generated from $u$ ($H_0$) or from a distribution in $\cC_L(\varepsilon)$ ($H_1$).

\subsection{Adaptive Binning Coincidence Test}

We now generalize SCT to the continuous problem specified above. Our goal is to preserve the adaptivity of SCT to the underlying alternative distribution under $H_1$. \\

The relation among the set of continuous distributions can still be visualized as an onion: $u$ and its $\varepsilon$-neighbors form the core, and alternative distributions in $\cC_L(\varepsilon)$ form layers according to their distances to $u$. The algorithm still aims to determine, sequentially over epochs, which layer the underlying distribution of the observed samples resides, starting from the outer-most and moving towards the core. The key question in the continuous problem is how to infer, from a coincidence type of statistic, whether the underlying distribution resides in the current layer. A straightforward answer to this question is discretization: a uniform binning of the support set $[0,1]$ with coincidence number defined with respect to the bin labels of the random samples. Much less obvious is the choice of the resolution level for the discretization, i.e., how finely to bin the continuum domain. A key rationale behind the proposed ABC (Adaptive Binning Coincidence) test is that the farther away the layer is from the core, the coarser the discretization is needed for inferring whether the underlying distribution resides in this layer. More specifically, not only the test can exit early when the realized distance $\gamma$ is favorable, but also the number of required samples for making the peeling decision is fewer due to a coarser discretization. In other words, ABC adapts to the unknown realized distance $\gamma$ by exiting both \emph{early} in the detection process and \emph{quickly} by using a lower resolution to take advantage of favorable alternative distributions.  \\

\begin{algorithm}
	\caption{Adaptive Binning Coincidence (ABC) Test}
	\label{alg:ABC}
	\begin{algorithmic}
		\STATE {\bfseries Input:} $\varepsilon, L$,  $\delta \in (0,1)$ 
		\STATE Set $k \leftarrow 1$, $t \leftarrow 0$, $\kappa \leftarrow 576\varepsilon^{-2}$, $\cS \leftarrow \{\}$
	    \WHILE{$k \leq  \kappa$ }
	    \STATE $m_k  \leftarrow \lceil c_{0} k^4 \log(k+2/\delta) \rceil$
	    \STATE $n_k \leftarrow \lceil \sqrt{c_{0}} k^3 \log(k+2/\delta) \rceil$
	    \STATE $\displaystyle \tau_k \leftarrow 9n_k\sqrt{\frac{\log(k+2/\delta)}{m_k}}$
	    \REPEAT
	    \STATE Obtain a sample $X_t$ 
	    \STATE $\cS \leftarrow \cS \cup X_t$
	    \STATE $t \leftarrow t + 1$
	    \UNTIL{$t == n_k$}
	    \IF{$Z_k := c_u(n_k; m_k) - K_1(\cS; m_k) > \tau_k$}
	    \STATE Output $\leftarrow H_1$
	    \STATE \textbf{break}
	    \ENDIF
	    \STATE $k \leftarrow k + 1$
	    \ENDWHILE
	    \IF{$k > \kappa$}
	    \STATE Output $\leftarrow H_0$
	    \ENDIF
	    \RETURN Output
	\end{algorithmic}
\end{algorithm}

We can now describe the ABC test, which proceeds in a similar epoch structure as SCT with two key modifications. First, the resolution of the discretization increases at a carefully chosen rate over epochs, and the number of samples taken in each epoch is adjusted accordingly. Specifically, let $m_k$ denote the number of discretization bins in epoch $k$, where $m_k$ increases at the rate of $k^4\log k$. The number of samples taken in epoch $k$ is $\Omega(\sqrt{m_k})$, which retains the same squared-root relation to the effective support size $m_k$ as in the discrete case. Second, the coincidence number $K_1(\cS)$ in each epoch is computed by rebinning all observed samples (including those obtained in previous epochs) based on the refined discretization $m_k$ in the current epoch.   A detailed description of the algorithm is given in Algorithm~\ref{alg:ABC}, where $K_1(\cS; m)$ denotes the coincidence number computed over the set $\cS$ of samples when the interval is uniformly divided into $m$ bins. Similarly, $c_u(n; m)$ denotes the expected coincidence number of $n$ samples from the uniform distribution with a support of $m$ bins. The constant\footnote{The constant $28212$ as a lower bound for $c_0$ arises from the conditions imposed during analysis. Please refer to Appendix~\ref{proof:theorem_sct_samp_complexity} for exact expressions for these conditions. We would also like to point out that the constants are not optimized. The analysis focuses on the order.} $c_{0} \geq \max\{ 28212 , m_0, 2L\}$, where $m_0$ is  defined in~Theorem~\ref{thm:sct_samp_complexity}. \\

\subsection{Sample Complexity}

The theorem below establishes the sample complexity of the ABC test and its adaptivity to the realized distance $\gamma$ under $H_1$.

\begin{theorem}
\begin{itemize}
    \item Under $H_1$ with an alternative distribution $p$ that is $\gamma$ away from $u$, the expected sample complexity of ABC is $\displaystyle O\left(\frac{1}{\gamma^6}  \log\left( \frac{1}{\gamma} + \frac{1}{\delta} \right)\right) $.
    \item Under $H_0$, the expected sample complexity of ABC is $\displaystyle O\left(\frac{1}{\varepsilon^6}  \log\left( \frac{1}{\varepsilon} + \frac{1}{\delta} \right)\right) $.
    \item Under both $H_1$ and $H_0$, the probability of correct detection under ABC is at least $1 - \delta$. \label{thm:abc_samp_complexity}
\end{itemize}
\end{theorem}

\begin{proof}
See Appendix~\ref{proof:theorem_abc_samp_complexity}.
\end{proof}

\subsection{Lower Bounds}

We now establish a lower bound on sample complexity for uniformity testing of Lipschitz continuous distributions. A commonly used approach to establishing a lower bound is to construct a set of distributions such that a randomly chosen distribution from the set is difficult to distinguish from the uniform distribution. More specifically, we construct  a class $\cQ$ of distributions of size $2^{N_{\varepsilon}}$ where $N_{\varepsilon} = L/2C\varepsilon$, for some constant $C > 0$. The class $\cQ$ is defined using a bijection from $2^{N_{\varepsilon}}$ bit vectors in $\{ -1, +1\}^{N_{\varepsilon}}$. Specifically, if $z = [z_0, z_1, \dots, z_{N_{\varepsilon}-1}] \in \{ -1, +1\}^{N_{\varepsilon}}$, then the density function $f_q$ of $q \in \cQ$ corresponding to $z$ is given by
\begin{align*}
    f_q(x) = \sum_{j = 0}^{N_{\varepsilon}-1} \left[f_{z_j} (x - 2Cj\varepsilon/L) + f_{-z_j} (x - C(2j+1)\varepsilon/L)\right],
\end{align*}
where the function $f_w(x)$ is given as
\begin{align*}
    f_w(x) & = \begin{cases} 1 + wLx & \text{ if } 0 \leq x \leq C\varepsilon/2L \\ 1 + w(C\varepsilon - Lx) & \text{ if } C\varepsilon/2L \leq x \leq C\varepsilon/L \\ 0 & \text{ otherwise. } \end{cases}
\end{align*}
It is not difficult to see that $f_q$ is a $L$-Lipschitz function for all $z$. We also show that, for an appropriate choice of the constant $C$, the $\ell_1$ distance of any such distribution $q$ is greater than $\varepsilon$ implying that all of them belong to $\cC_L(\varepsilon)$. Then by using LeCam's Lemma~\citep{LeCam1986}, we can show that for fewer than $\Omega(\varepsilon^{-2.5})$ samples, the probability of not being able to identify a uniformly chosen distribution from $\cQ$ from the uniform distribution is at least $0.1$, yielding us the required lower bound. This is formalized in the following theorem. \\

\begin{theorem}
Consider the problem of uniformity testing where the alternative distributions belong in $\cC_L(\varepsilon)$. For all tests guaranteeing that the probabilities of false alarm and miss detection are capped below $0.1$, their sample complexity is $\Omega(\varepsilon^{-2.5})$.
\label{thm:lower_bound}
\end{theorem}

Evidently, there is a significant gap between the lower bound on sample complexity and the sample complexity guarantees offered by ABC. This gap, we believe is rooted in that coincidence statistic is informative only in sparse regimes where the number of samples is of sublinear order of the support size. We conjecture that the lower bound is indeed tight, and the gap to the lower bound is unavoidable for tests using the coincidence statistic. In other words, we conjecture that while coincidence statistic is sufficient for achieving order-optimal sample complexity in the discrete case, it ceases to remain so in the continuous case. An interesting question here is to explore what would be a sufficient statistic and whether requiring the entire histogram is indeed necessary. Moreover, developing and analyzing sequential variants based on such test statistics also appears to an interesting future direction.

\section{Simulations}

\begin{figure}
    \centering
    \includegraphics[scale=0.6]{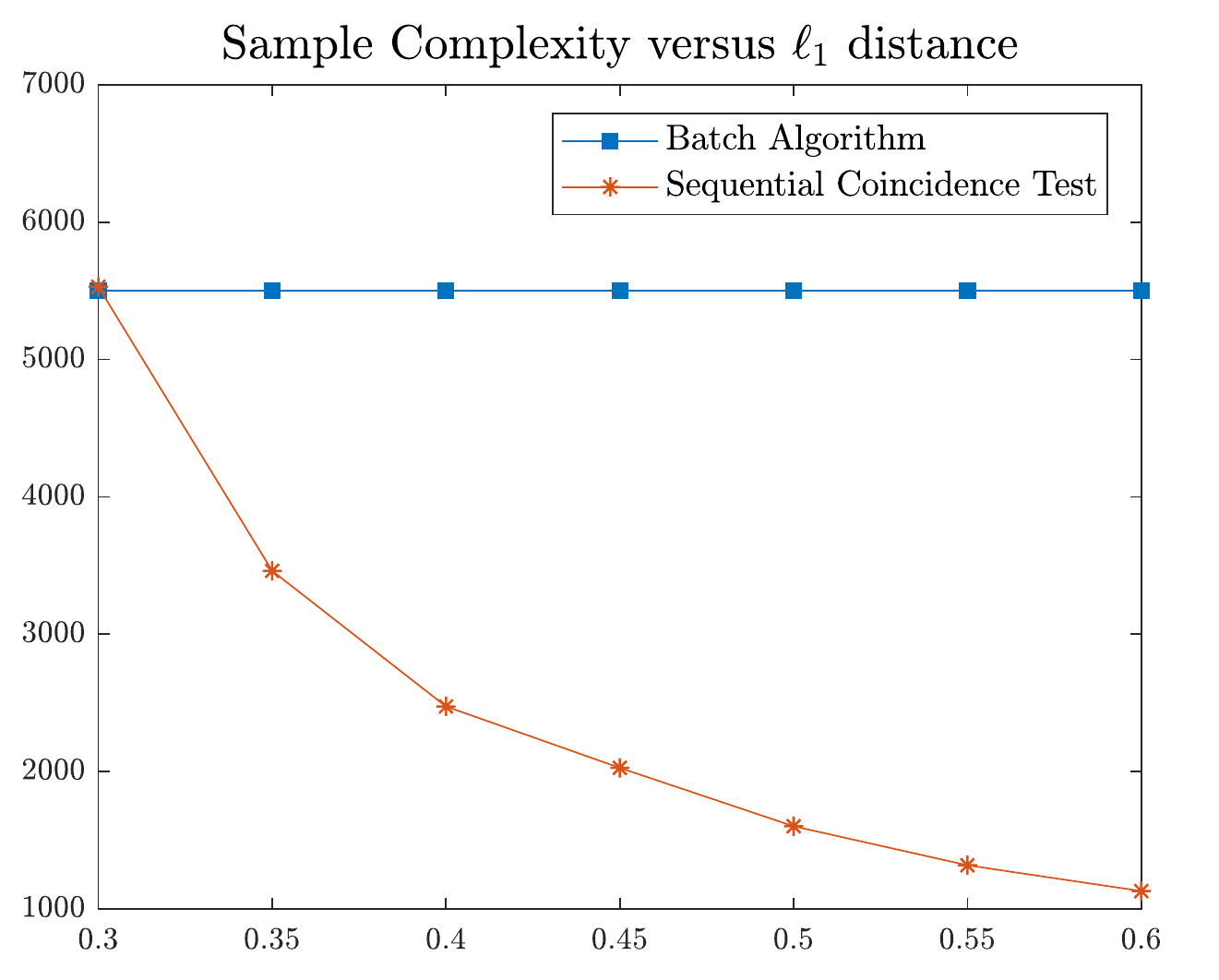}
    \caption{Sample Complexity versus $\ell_1$ distance. The plot demonstrates how the proposed Sequential Coincidence Test adapts to the underlying distribution offering better sample complexities.}
    \label{fig:sct_plot}
\end{figure}

We corroborate our theoretical findings by testing the algorithms empirically. We consider the problem of binary hypothesis testing described in~Section~\ref{sub:discrete_prob_des} with $m = 20000$ and $\varepsilon = 0.3$. In the experimental setup, we consider a set of distributions in $\cC(\varepsilon)$ parametrized by $\gamma$ and whose probability mass functions are given as follows:
\begin{align*}
    p^{\gamma}_{2i-1} = (1+ \gamma)/m; \ p^{\gamma}_{2i} = (1 - \gamma)/m,
\end{align*}
for $i = \{1,2, \dots, m/2\}$. It can be noted that the $\ell_1$ distance of $p^{\gamma}$ from the uniform distribution is $\gamma$. We consider $7$ distributions corresponding to the values of $\gamma \in \{0.3, 0.35, 0.4, 0.45, 0.5, 0.55, 0.6\}$. \\

For each of these distributions, we run the fixed sample size coincidence test proposed by~\cite{Paninski2008} and the Sequential Coincidence Test proposed in this work. For both the algorithms, the threshold is set to obtain an accuracy of at least $0.8$. The constants in the threshold for both the algorithms are optimized using grid search to give the best sample complexity. For the sequential coincidence test, the expected sample complexity is obtained by taking an average over $1000$ Monte Carlo runs. We then plot the number of samples taken by each algorithm against the $\ell_1$ distance of underlying distribution. The resulting plot is shown in Figure~\ref{fig:sct_plot}. As one would expect, the sample complexity for the coincidence test is the same for all the distributions. However, the sample complexity of the Sequential Coincidence Test adapts to the $\ell_1$ distance of the underlying distribution and decreases as the $\ell_1$ distance increases, demonstrating the benefit of the proposed approach.

\bibliography{citations}
\bibliographystyle{plainnat}	

\appendix

\newpage

\section{Proof of Theorem~\ref{thm:sct_samp_complexity}}
\label{proof:theorem_sct_samp_complexity}

The central idea of this proof is to establish bounds on the moment generating function of the $K_1$ statistic. Once we have obtained the bound on MGF, the final result follows directly by an application of Markov's inequality and union bound. Hence, we first focus on bounding the MGF of $K_1$ and then would work on establishing bounds on probability of error and sample complexity. \\

Note that $K_1$ is a sub-Gaussian random variable since it is bounded, so we can obtain an upper bound on its moment generating function using this property. However, the variance proxy obtained by just using the boundedness of $K_1$ is not tight enough for our purposes. Thus, in order to obtain tight bounds on the variance proxy we employ an approach similar to the one used in \cite{Huang2012}. Since their focus was mainly on asymptotic analysis of the coincidence test for obtained error rates, we need to appropriately modify the approach to obtain bounds for finite sample regime. \\

Before we begin the proof, we set up some additional notation that we will be using throughout the proof. The underlying discrete distribution is denoted by $p$. If $\cS$ denotes a set of i.i.d. samples from $p$ such that $|\cS| = n$, then with a slight abuse of notation, we denote $K_1(\cS)$ as $K_1(n)$ since the relevant quantity throughout the analysis will be the size of the sample. We derive the bounds for a general $n \geq \sqrt{m}$ number of samples, followed by substituting the particular choices later. Lastly, we assume that $n/m \leq \varepsilon^2/1536$. While the sparse regime is required to use the coincidence test, the additional dependence on $\varepsilon$ is a result of the particular technique being employed to bound the MGF. We conjecture that requirement can be relaxed by using better analysis tools and techniques. \\

From \cite{Huang2012} (eqn. (38)), we have that the MGF satisifes the following equation:
\begin{align}
    \E_{p}\left[ \exp \left( \theta \tilde{K}_1(n) \right)\right] = e^{-\theta n}\frac{n!}{2 \pi i} \oint g(\lambda) d \lambda, 
    \label{eqn:MGF_defn}
\end{align}
where $\tilde{K}_1(n) = K_1(n) - n$ and
\begin{align*}
    g(\lambda) = e^{\lambda} \prod_{j = 1}^m \left(1 - \lambda p_j e^{-\lambda p_j} + \lambda p_j e^{-\lambda p_j} e^{\theta} \right) \frac{1}{\lambda^{n+1}}.
\end{align*}
Throughout the analysis, we assume that $\theta \in [-0.4 ,0.4]$. The integral in~\eqref{eqn:MGF_defn} is estimated using the saddle point method~\citep{de1981}. This approach consists of two steps. In the first step, a particular contour around $\lambda = 0$ is chosen to carry out the integration. The choice of this contour is such that $g(\lambda)$ behaves violently along it, i.e., $g(\lambda)$ is very large for a very small interval and significantly smaller on the rest of the contour, similar to a dirac delta function. This violent behaviour allows one to approximate the integral by only evaluating it on the small interval where the function value is very large. Such a contour is usually obtained by identifying a saddle point of $g(\lambda)$, a point where the derivative of $g$ goes to zero and choosing to contour to pass through this point. The second step involves estimating the integral along the contour using the Laplace method.  \\

To choose the contour, we first differentiate $g$ to find the saddle point. On differentiating, we obtain,
\begin{align*}
    g'(\lambda) = g(\lambda) \left( \sum_{j = 1}^m \left( \frac{p_j(e^{\theta} - 1 + e^{\lambda p_j})}{\lambda p_j(e^{\theta} - 1) + e^{\lambda p_j}} \right) - \frac{n + 1}{\lambda} \right).
\end{align*}

Instead of exactly choosing the minimizer, we choose a point very close to it, defined as the solution to the following equation:
\begin{align}
    \frac{\lambda p_j(e^{\theta} - 1 + e^{\lambda p_j})}{\lambda p_j(e^{\theta} - 1) + e^{\lambda p_j}} = n. \label{eqn:saddle_pt_eqn}
\end{align}

It is established in~\cite{Huang2012} that the above equation has a unique non-negative solution, which we denote using $\lambda_0$. Furthermore, it satisfies the following inequalities:
\begin{align}
    \text{For } \theta \geq 0: & \ \ \ \  ne^{-\theta} \leq \lambda_0 \leq n( 1 + e^{-1}(e^{\theta} - 1))  \nonumber \\
    \text{For } \theta < 0: & \ \ \ \  n( 1 + e^{-1}(e^{\theta} - 1)) \leq \lambda_0 \leq ne^{-\theta}. \label{eqn:lambda_0_relations}
\end{align}
From the above inequalities, it is clear that $\lambda_0 = O(n)$. We can also obtain a more refined approximation of the relation between $\lambda_0$ and $n$ as follows. We first split the set of symbols into two set based on the magnitude of their probabilities and obtain the approximation separately for these sets. In particular, we define $\cW = \{j : p_j \geq 8/m \}$. Using this, we define $\beta(p) = \sum_{j \in \cW} p_j$. \\

We first focus on the symbols that not do not belong in $\cW$, i.e., $j \notin \cW$. For these symbols, the following set of inequalities hold.
\begin{align}
    \lambda_0 p_j e^{\theta} + \lambda_0^2 p_j^2 (1 - e^{2\theta}) + l(\theta) \lambda_0^3 p_j^3  \leq \frac{\lambda_0 p_j (e^\theta - 1 + e^{\lambda p_j})}{\lambda_0 p_j (e^\theta - 1) + e^{\lambda_0 p_j}} \leq \lambda_0 p_j e^{\theta} + \lambda_0^2 p_j^2 (1 - e^{2\theta}) + u(\theta) \lambda_0^3 p_j^3 , \label{eqn:bounds_for_W_bar}
\end{align}
where $l(\theta)$ and $u(\theta)$ are the following functions of $\theta$:
\begin{align*}
    l(\theta) & = \begin{cases} \frac{1}{6}\left(6e^{3\theta} - 9e^{\theta} + 3 \right) & \text{ if } \theta <0  \\ 0 & \text{ if } \theta \geq 0 \end{cases} \\
    u(\theta) & = \begin{cases} 0 & \text{ if } \theta <0  \\ \frac{1}{6}\left(6e^{3\theta} - 9e^{\theta} + 3 \right) & \text{ if } \theta \geq 0 \end{cases}.
\end{align*}
For the symbols in $\cW$, we have,
\begin{align*}
    \frac{\lambda_0 p_j (e^\theta - 1 + e^{\lambda p_j})}{\lambda_0 p_j (e^\theta - 1) + e^{\lambda_0 p_j}} = D_j \lambda_0 p_j e^{\theta}, 
\end{align*}
where $\displaystyle D_j = \frac{e^{-\theta} + e^{-\lambda_0 p_j}(1 - e^{-\theta})}{1 + \lambda_0 p_j e^{-\lambda_0 p_j} (e^{\theta} - 1)}$.  \\

On plugging these approximations for different regimes in~\eqref{eqn:saddle_pt_eqn}, we obtain that $\lambda_0$ is of the form, $\displaystyle \lambda_0 = \dfrac{ne^{-\theta}(1 + w)}{1 + \sum_{j \in \cW}p_j(D_j - 1)}$, where $w$ is the small approximating factor, on which we will obtain bounds. For brevity, we define $D_* := 1 + \sum_{j \in \cW}p_j(D_j - 1)$. Lastly, we can obtain the range on the ratio $(1 + w)/D_*$ from the relations in~\eqref{eqn:lambda_0_relations}.  \\

On summing the lower bound over $j$ in~\eqref{eqn:bounds_for_W_bar}, we obtain,
\begin{align*}
    n & \geq \sum_{j \in \cW} D_j \lambda_0 p_j e^{\theta} +   \sum_{j \notin \cW} \left[ \lambda_0 p_j e^{\theta} + \lambda_0^2 p_j^2 (1 - e^{2\theta})  + l(\theta) \lambda_0^3  p_j^3 \right]   \\
    n & \geq \lambda_0 D_*e^{\theta} +   \lambda_0^2 \sum_{j \notin \cW}   p_j^2 (1 - e^{2\theta})  + l(\theta) \lambda_0^3 \sum_{j \notin \cW} p_j^3   \\
    n & \geq  n(1 + w) +   \frac{n^2}{{D}_*^2} (1 + w)^2 \sum_{j \notin \cW}   p_j^2 (e^{-2\theta} - 1) + l(\theta) n^3 \left( \frac{1 + w}{D_*}\right)^3 \sum_{j \notin \cW} p_j^3 
\end{align*}
Consequently, we have, 
\begin{align}
    w & \leq \frac{n}{{D}_*^2} (1 + w)^2 \sum_{j \notin \cW}   p_j^2 (1 - e^{-2\theta}) - l(\theta) n^2 \left( \frac{1 + w}{D_*}\right)^3 \sum_{j = 1}^m p_j^3 . \label{eqn: w_upper_bound}
\end{align}

Similarly, summing over the upper bound yields us the following relation.
\begin{align*}
    n & \leq \sum_{j \in \cW} D_j \lambda_0 p_j e^{\theta} +   \sum_{j \notin \cW} \left[ \lambda_0 p_j e^{\theta} + \lambda_0^2 p_j^2 (1 - e^{2\theta}) + u(\theta) \lambda_0^3 p_j^3 \right]  \\
    n & \leq \lambda_0 {D}_* e^{\theta} +   \lambda_0^2 \sum_{j \notin \cW}   p_j^2 (1 - e^{2\theta}) + u(\theta) \lambda_0^3 \sum_{j \notin \cW}   p_j^3  \\
    n & \leq n(1 + w) + \frac{n^2}{{D}_*^2} (1 + w)^2 \sum_{j \notin \cW}   p_j^2 (e^{-2\theta} - 1)  + u(\theta) n^3 \left( \frac{1 + w}{D_*} \right)^3 \sum_{j \notin \cW}   p_j^3   \\
\end{align*}
Consequently, we have, 
\begin{align}
     w \geq \frac{n(1 + w)^2}{{D}_*^2} \left( \sum_{j \notin \cW}   p_j^2 \right) (1 - e^{-2\theta}) - u(\theta) n^2 \left( \frac{1 + w}{D_*} \right)^3 \sum_{j \notin \cW}   p_j^3 . \label{eqn: w_lower_bound}
\end{align}

With these relations at our disposal, we now move on to evaluating the integral in~\eqref{eqn:MGF_defn} along the closed contour $\lambda = \lambda_0 e^{i \psi}$ for $\psi \in [-\pi, \pi]$. The integral reads as
\begin{align*}
    \E_{p}\left[ \exp \left( \theta \tilde{S}^*_n \right) \right] & = e^{-\theta n} \frac{n!}{2\pi} \int_{-\pi}^{\pi} g(\lambda_0 e^{i \psi})  \lambda_0 e^{i \psi} \ \mathrm{d} \psi \\
    & = \frac{n!}{2\pi} \lambda_0^{-n} e^{-\theta n}  \text{Re} \left[ \int_{-\pi}^{\pi} h(\psi) \ \mathrm{d} \psi \right],
\end{align*}
where 
\begin{align*}
    h(\psi) = e^{-in\psi} \prod_{j = 1}^m \left( \lambda_0 p_j (e^{\theta} - 1) e^{i \psi} + e^{\lambda_0 p_j e^{i \psi}} \right).
\end{align*}
Instead of $h(\psi)$, we focus on bounding $H(\psi) = \log(h(\psi))$ since it is easier to deal with sums than products. We have,
\begin{align*}
     H(\psi) & = - in\psi + \sum_{j=1}^m \log \left(  \lambda_0 p_j (e^\theta - 1) e^{i \psi} + e^{\lambda_0 p_j e^{i \psi}} \right).
\end{align*}

We split the integral into three parts by evaluating it over three different ranges, i.e., $[-\pi, -\pi/2)$, $[-\pi/2, \pi/2]$ and $(\pi/2, \pi]$. We first consider the integral over  $[-\pi/2, \pi/2]$. This is the region where the integrand behaviour behaves violently and thus will correspond to the dominant term. We have, 
\begin{align*}
     \text{Re}(H(\psi)) & = \sum_{j=1}^m \text{Re} \left[ \log \left(  \lambda_0 p_j (e^\theta - 1) e^{i \psi} + e^{\lambda_0 p_j e^{i \psi}} \right) \right] \\
     & = \sum_{j=1}^m \text{Re} \left[ \log( e^{\lambda_0 p_j e^{i \psi}}) +   \log \left( 1 + \lambda_0 p_j (e^\theta - 1) e^{i \psi} e^{-\lambda_0 p_j e^{i \psi}} \right) \right] \\
     & \leq  \sum_{j=1}^m  \left[ \lambda_0 p_j \cos(\psi) +   \log \left( 1 + \lambda_0 p_j   e^{-\lambda_0 p_j \cos(\psi)} (e^\theta - 1) \right) \right].
\end{align*}

We define a function $G(\psi; u)$ that corresponds to the general form of the RHS in the above expression. That is for $u \geq 0$, 
\begin{align*}
    G(\psi; u) =  u \cos(\psi) +   \log \left( 1 + u   e^{-u \cos(\psi)} (e^\theta - 1) \right).
\end{align*}
Note that $G'(0; u) = 0$ and $G''(0; u) \leq -0.4u$ for $\psi \in [-\pi/2, \pi/2]$. Thus, by using mean value theorem it can be shown that $G(\psi; u)$ satisfies $G(\psi; u) \leq G(0; u) - 0.2 u \psi^2 $ for $\psi \in [-\pi/2, \pi/2]$ for all $u \geq 0$. On plugging this inequality in previous one, we have,
\begin{align*}
     \text{Re}(H(\psi)) & \leq  \sum_{j=1}^m  G(\psi; \lambda_0 p_j) \\
     & \leq  \sum_{j=1}^m  \left[ G(0; \lambda_0 p_j) - 0.2 \lambda_0 p_j \psi^2 \right] \\
     & \leq  \sum_{j=1}^m  \left[ \lambda_0 p_j  +   \log \left( 1 + \lambda_0 p_j   e^{-\lambda_0 p_j } (e^\theta - 1) \right) \right] - 0.2 \lambda_0 \psi^2  \\
     & \leq  H(0) - 0.2 \lambda_0 \psi^2.
\end{align*}

Using this relation, we can establish the following bound.
\begin{align}
    \text{Re} \left[ \int_{-\pi/2}^{\pi/2} h(\psi) \ \mathrm{d} \psi \right] & = \text{Re} \left[ \int_{-\pi/2}^{\pi/2} \exp(H(\psi)) \ \mathrm{d} \psi \right] \nn \\
    & \leq  \int_{-\pi/2}^{\pi/2} |\exp(H(\psi))| \ \mathrm{d} \psi \nn \\
    & \leq  \int_{-\pi/2}^{\pi/2} \exp(\text{Re}(H(\psi))) \ \mathrm{d} \psi \nn \\
    & \leq  e^{H(0)} \int_{-\pi/2}^{\pi/2} e^{-0.2\lambda_0 \psi^2} \ \mathrm{d} \psi \nn \\
    & \leq  e^{H(0)} \int_{-\infty}^{\infty} e^{-0.2\lambda_0 \psi^2} \ \mathrm{d} \psi \nn \\
    & \leq e^{H(0)} \sqrt{\frac{\pi}{0.1\lambda_0}}. \label{eqn:pi2_to_pi2_bound}
\end{align}

For $\psi \in [-\pi, -\pi/2] \cup [\pi/2, \pi]$, we have $|e^{\lambda_0 p_j e^{i\psi}}| \leq 1$. Consequently, $|\lambda_0 p_j (e^\theta - 1) e^{i \psi} + e^{\lambda_0 p_j e^{i \psi}}| \leq 1 + \lambda_0 p_j(e^{\theta} - 1)$. Thus, we have,
\begin{align*}
    \text{Re}(H(\psi)) & = \sum_{j=1}^m \text{Re} \left[ \log \left(  \lambda_0 p_j (e^\theta - 1) e^{i \psi} + e^{\lambda_0 p_j e^{i \psi}} \right) \right] \\
    & \leq  \sum_{j=1}^m \log \left(  1 + \lambda_0 p_j |(e^\theta - 1)| \right)  \\
     & \leq  \sum_{j=1}^m  \lambda_0 p_j |(e^\theta - 1)|   \\
     & \leq   \lambda_0  |(e^\theta - 1)|.
\end{align*}

Using this bound, we can bound the following integral.
\begin{align*}
    \text{Re} \left[ \int_{-\pi}^{-\pi/2} h(\psi) \ \mathrm{d} \psi \right] & = \text{Re} \left[ \int_{-\pi}^{-\pi/2} \exp(H(\psi)) \ \mathrm{d} \psi \right] \\
    & \leq  \int_{-\pi}^{-\pi/2} |\exp(H(\psi))| \ \mathrm{d} \psi \\
    & \leq  \int_{-\pi}^{-\pi/2} \exp(\text{Re}(H(\psi))) \ \mathrm{d} \psi \\
    & \leq   \int_{-\pi}^{-\pi/2} e^{\lambda_0 |e^{\theta} - 1|} \ \mathrm{d} \psi \\
    & \leq \frac{\pi}{2} e^{\lambda_0 |e^{\theta} - 1|}.
\end{align*}
We can similarly bound the integral for $[\pi/2, \pi]$. On combining all of them, we have, 
\begin{align}
    \E_{p}\left[ \exp \left( \theta \tilde{K}_1(n) \right) \right]  & = \frac{n!}{2\pi} \lambda_0^{-n} e^{-\theta n}  \text{Re} \left[ \int_{-\pi}^{\pi} h(\psi) \ \mathrm{d} \psi \right] \nn \\
    & \leq \frac{n!}{2\pi} \lambda_0^{-n} e^{-\theta n}   \left( e^{H(0)} \sqrt{\frac{\pi}{0.1\lambda_0}} +  \pi e^{\lambda_0 |(e^{\theta} - 1)|} \right)\nn \\
    & \leq \frac{n!}{2\pi n^n} \left(\frac{D_*}{1 + w}  \right)^n \left( e^{H(0)} \sqrt{\frac{\pi}{0.1\lambda_0}} +  \pi e^{\lambda_0 |(e^{\theta} - 1)|} \right). \label{eqn:mgf_upper_bound}
\end{align}
Since $\theta \in [-0.4, 0.4]$, the second term can be upper bounded as $e^{\lambda_0 |(e^{\theta} - 1)|} \leq e^{0.7n}$.  \\

Using the above bound on the moment generating function, we can evaluate bounds on the probability of error. We begin with bounding the probability of false alarm, that is, the underlying distribution is uniform and we fail to detect it. For this scenario, the following relation holds for $\theta < 0$, where $\tau_n$ denotes the threshold when $n$ samples have been taken and $P_e(n)$ denotes the error probability at that time instant.
\begin{align*}
    P_e(n) & = \Pr( K_1(n) - \E_u[K_1(n)] < -\tau_n) \\
    & = \Pr( \theta(K_1(n) - \E_u[K_1(n)]) > -\theta \tau_n) \\
    & = \Pr( \exp(\theta(K_1(n) - \E_u[K_1(n)])) > \exp(-\theta \tau_n)) \\
    & \leq \E \left[ \exp(\theta(K_1(n) - \E_u[K_1(n)])) \right] e^{\theta \tau_n} \\
    & \leq \E \left[ \exp(\theta(K_1(n) - n)) \right] \exp \left(\theta (n - \E_u[K_1(n)] + \tau_n) \right) \\
    & \leq \E \left[ \exp(\theta \tilde{K}_1(n))) \right] \exp \left( \theta\tau_n + \theta \left( \frac{n(n-1)}{m} - \frac{n}{2} \left( \frac{n-1}{m} \right)^2 \right)\right)   \\
    & \leq \frac{n!}{2\pi n^n} \left(\frac{D_*}{1 + w}  \right)^n \left( e^{H(0)} \sqrt{\frac{\pi}{0.1\lambda_0}} +  \pi e^{0.7n} \right) \exp \left( \theta\tau_n + \theta \left( \frac{n(n-1)}{m} - \frac{n}{2} \left( \frac{n-1}{m} \right)^2 \right)\right) \\
    & \leq \frac{n! e^n}{2\pi n^n} \left(\frac{D_*}{1 + w}  \right)^n \left( e^{H(0) - n} \sqrt{\frac{\pi}{0.1\lambda_0}} +  \pi e^{-0.3n} \right) \exp \left( \theta\tau_n + \theta \left( \frac{n(n-1)}{m} - \frac{n}{2} \left( \frac{n-1}{m} \right)^2 \right)\right)
\end{align*}

Since the underlying distribution is uniform, $\cW = \emptyset$ and consequently $D_* = 1$. The expression $H(0)$ can be bounded as follows for $\theta < 0$.
\begin{align*}
    H(0) & = \sum_{j=1}^m \log \left(  \lambda_0 p_j (e^\theta - 1) + e^{\lambda_0 p_j} \right) \\
    & \leq \sum_{j =1 }^m \left[ \lambda_0 p_j e^{\theta} +  \frac{1}{2} \lambda_0^2 p_j^2 (1 - e^{2\theta}) \right].
\end{align*}

Using the above relation, we first bound the first term, which we denote by $T_1$, as that is the dominant term in the expression. We have,
\begin{align*}
    T_1 & = \left(\frac{D_*}{1 + w}  \right)^n \exp \left( H(0) - n + \theta\tau_n + \theta \left( \frac{n(n-1)}{m} - \frac{n}{2} \left( \frac{n-1}{m} \right)^2 \right)\right) \\
    & \leq \left({1 + w}  \right)^{-n} \exp \left( \sum_{j =1 }^m \left[ \lambda_0 p_j e^{\theta} +  \frac{1}{2} \lambda_0^2 p_j^2 (1 - e^{2\theta}) \right]  - n + \theta\tau_n + \theta \left( \frac{n(n-1)}{m} - \frac{n}{2} \left( \frac{n-1}{m} \right)^2 \right)\right) \\
    & \leq \left(1+ w  \right)^{-n} \exp \left(  \lambda_0  e^{\theta} +  \frac{n^2}{2m} ( 1+w)^2 (e^{-2\theta} -1)   - n + \theta\tau_n + \theta \left( \frac{n(n-1)}{m} - \frac{n}{2} \left( \frac{n-1}{m} \right)^2 \right)\right) \\
    & \leq \exp \left(    \frac{n^2}{2m} (-2\theta + 4\theta^2)  + \theta\tau_n + \theta \left( \frac{n(n-1)}{m} - \frac{n}{2} \left( \frac{n-1}{m} \right)^2 \right) + n(w - \log(1 + w))\right) \\
    & \leq \exp \left(  \frac{2n^2\theta^2}{m}   +  \theta \left(\tau_n -\frac{n}{m} - \frac{n}{2} \left( \frac{n}{m} \right)^2 \right)  + 2nw^2 \right).
\end{align*}

Using the condition on the ratio of $n$ and $m$, we have
\begin{align*}
    \frac{n}{m} + \frac{n^3}{2m^2} &\leq \frac{n^3}{m^2} \\
    &\leq \frac{n^2\epsilon^2}{1536m}.
\end{align*}

Note that for the particular choice of $\tau$ and corresponding decision instant in the Sequential Coincidence Test satisfies $n^2 \varepsilon^2 /1536m \leq \tau_n/75$ for all epochs $k$. Lastly, using~\eqref{eqn: w_upper_bound}, we have, $2nw^2 \leq 2n \left( 3.76\frac{n}{m} (1 - e^{-2\theta})\right)^2 \leq \frac{114n^2 \epsilon^2 \theta^2}{1536m} $. On combining these two bounds with the bound on $T_1$, we obtain,
\begin{align*}
    T_1 & \leq \exp \left(  \frac{n^2\theta^2}{m} \left( 2 + \frac{114\epsilon^2}{1536} \right)+  \frac{74\theta \tau_n}{75} \right).
\end{align*}
On plugging $\theta = -\frac{74m\tau_n}{375n^2}$, we obtain
\begin{align*}
    T_1 & \leq \exp \left(  -\frac{5476}{56250} \cdot \frac{m\tau_n^2}{n^2} \right).
\end{align*}

On evaluating the above expression at decision instant for epoch $k$, that is, at $n = n_k$ and $\tau = \tau_k$, we obtain
\begin{align*}
    T_1 & \leq \exp \left(  -3 \log(k + 2/{\delta}) \right) \\
    & \leq \frac{1}{(k+2/\delta)^3}.
\end{align*}

Let us consider the second term of the expression with the above value of $\theta$, denoted by $T_2$. We have,
\begin{align*}
    T_2 & = \left(\frac{D_*}{1 + w}  \right)^n \exp\left( -0.3n + \theta\tau_n + \theta \left( \frac{n(n-1)}{m} - \frac{n}{2} \left( \frac{n-1}{m} \right)^2 \right)\right) \\
    & \leq \exp\left( -0.3n  - \frac{n \theta}{2} \left( \frac{n-1}{m} \right)^2 - n \log(1 + w) \right) \\
    & \leq \exp\left( -0.25n \right),
\end{align*}
where the last expression follows using the value of $\theta$ and bound on $w$ obtained from~\eqref{eqn: w_lower_bound}. \\

On combining both the expressions, we obtain the following result for the probability of error at the decision instant corresponding to epoch $k$.
\begin{align*}
    P_e(n_k) & \leq \frac{n_k! e^{n_k}}{2\pi {n_k}^{n_k}} \left( \frac{1}{(k+2/\delta)^3} \cdot \sqrt{\frac{18\pi}{n_k}} +  \pi e^{-0.25n_k} \right)  \\
    & \leq e^{1/12n_k}   \left( \frac{3}{(k+2/\delta)^3}  +  \sqrt{\frac{n_k\pi}{2}} e^{-0.25n_k} \right) \\
\end{align*}
In the last step above we have used Stirling's approximation. Using the above expression for $P_e(n_k)$, we can obtain the probability for false alarm, which we denote by $\Pr(\text{err}|H_0)$. 
\begin{align*}
	\Pr(\text{err}|H_0) & = \Pr \left( \bigcup_{k = 1}^{\kappa} \left\{ \E_u[K_1(n_k)] - K_1(n_k) > \tau_{k} \right\} \right) \\
	& \leq \sum_{k = 1}^{\kappa} \Pr \left(  K_1^{(n_k)} - \E_u[K_1^{(n_k)}] < -\tau_k \right) \\
	& \leq \sum_{k = 1}^{\kappa} P_e(n_k) \\
	& \leq \sum_{k = 1}^{\kappa} e^{1/12n_k}   \left( \frac{3}{(k+2/\delta)^3}  +  \sqrt{\frac{n_k\pi}{2}} e^{-0.25n_k} \right) \\
	& \leq \left( 1 + \frac{1}{6\sqrt{m}} \right) \sum_{k = 1}^{\kappa}    \left( \frac{3}{(k+2/\delta)^3}  +  \sqrt{\frac{\sqrt{m}\pi}{2}} e^{-0.25\sqrt{m}} \right) \\
	& \leq \left( 1 + \frac{1}{6\sqrt{m}} \right)     \left( 3 \int_0^{\infty} \frac{1}{(x+2/\delta)^3} \ dx  +  \sqrt{\frac{\sqrt{m}\pi}{2}} \frac{96 e^{-0.25\sqrt{m}}}{\varepsilon^2} \right) \\
    & \leq \left( 1 + \frac{1}{6\sqrt{m}} \right)     \left( 3 \left( \frac{\delta}{2} \right)^2  +  \sqrt{\frac{\sqrt{m}\pi}{2}} \frac{112 e^{-0.25\sqrt{m}}}{\varepsilon^2} \right),
\end{align*}
where the fifth line follows by noting that $e^x \leq 1 + 2x$ for $x \leq 1$, $\sqrt{x}e^{-x/4}$ is decreasing for $x > 2$ and $n_k \geq \sqrt{m}$. On plugging in the lower bound for $m$, we obtain that the above expression is less than $\delta$. \\

Having obtained the bounds for false alarm, we now apply a similar process for bounding the probability of miss detection. In this case, the underlying distribution $p$ belongs to $\cC(\varepsilon)$ such that $\| p - u \|_1 = \gamma > \varepsilon$. For this scenario, the following relation holds for $\theta \geq 0$, with $\tau_n$ and $P_e(n)$ defined in a similar way as the previous case.
\begin{align*}
    \Pr(\text{err}) & = \Pr( \E_u[K_1(n)] - K_1(n) < \tau_n) \\
    & = \Pr( K_1(n) - \E_u[K_1(n)] > - \tau_n) \\
    & = \Pr( \theta(K_1(n) - \E_u[K_1(n)]) > -\theta \tau_n) \\
    & = \Pr( \exp(\theta(K_1(n) - \E_u[K_1(n)]) > \exp(-\theta \tau_n)) \\
    & \leq \E \left[ \exp(\theta(K_1(n) - \E_u[K_1(n)]) \right] e^{\theta \tau_n} \\
    & \leq \E \left[ \exp(\theta(K_1(n) - n)) \right] \exp \left(\theta (n - \E_u[K_1(n)] + \tau_n) \right) \\
    & \leq \E \left[ \exp(\theta(K_1(n) - n)) \right] \exp \left( \theta \left(\tau_n + \frac{n^2}{m} \right) \right) \\
    & \leq \frac{n!e^n}{2\pi n^n} \left(\frac{D_*}{1 + w}  \right)^n  \left( e^{H(0) -n } \sqrt{\frac{\pi}{0.1 \lambda_0 }} +  \pi e^{-0.3n} \right)  \exp \left( \theta \left(\tau_n + \frac{n^2}{m} \right) \right).
\end{align*}

For this scenario, $H(0)$ can be bounded as follows for $\theta \geq 0$.
\begin{align*}
    H(0) & = \sum_{j=1}^m \log \left(  \lambda_0 p_j (e^\theta - 1) + e^{\lambda_0 p_j} \right) \\
    & \leq \sum_{j \notin \cW} \left[\lambda_0 p_j e^{\theta} +  \frac{1}{2} \lambda_0^2 p_j^2  (1 - e^{2\theta})  + \frac{\theta e^{3\theta}}{2} \lambda_0^3 p_j^3  \right]   + \sum_{j \in \cW} \left(\lambda_0 p_j e^{\theta} +  \lambda_0 p_j (1 - e^{\lambda_0 p_j})(1 - e^{\theta})\right)
\end{align*}

Once again, we focus on the first term in the bound of $P_e(n)$, which we denote by $T_1'$, as that is the dominant term. Using the relation on $H(0)$ obtained from above, we have,
\begin{align*}
    T_1' & = \left(\frac{D_*}{1 + w}  \right)^n \exp \left( H(0) - n + \theta \left(\tau_n + \frac{n^2}{m} \right)\right) \\
    & \leq \left(\frac{D_*}{1 + w}  \right)^n \exp \left( - n + \theta \left(\tau_n + \frac{n^2}{m} \right) \right) \times \\
    & \ \ \ \exp\left( \sum_{j \notin \cW}\lambda_0 p_j e^{\theta} +  \frac{1}{2}\sum_{j \notin \cW} \lambda_0^2 p_j^2  (1 - e^{2\theta})  + \frac{\theta e^{3\theta}}{2} \sum_{j \notin \cW} \lambda_0^3 p_j^3    + \sum_{j \in \cW} \left(\lambda_0 p_j e^{\theta} +  \lambda_0 p_j (1 - e^{\lambda_0 p_j})(1 - e^{\theta})\right)     \right)  \\
    & \leq \left(\frac{D_*}{1 + w}  \right)^n \exp \left( - n + \theta \left(\tau_n + \frac{n^2}{m} \right) \right) \times \\
    & \ \ \ \exp\left( \lambda_0  e^{\theta} +  \frac{1}{2}\sum_{j \notin \cW} \lambda_0^2 p_j^2  (1 - e^{2\theta})  + \frac{\theta e^{3\theta}}{2} \sum_{j \notin \cW} \lambda_0^3 p_j^3    + \sum_{j \in \cW} \left( \lambda_0 p_j (1 - e^{\lambda_0 p_j})(1 - e^{\theta})\right)  \right)  \\
    & \leq \left(\frac{D_*}{1 + w}  \right)^n \exp \left( - n + \theta \left(\tau_n + \frac{n^2}{m} \right) \right) \times \\
    & \ \ \ \exp\left( \frac{n(1 + w)}{D_*} +  \frac{n^2}{2} \left( \frac{1 + w}{D_*}\right)^2 \sum_{j \notin \cW} p_j^2  (e^{-2\theta} - 1)  + \frac{32\theta n^3}{m^2} \left( \frac{1 + w}{D_*}\right)^3   + \sum_{j \in \cW} \left( \lambda_0 p_j (1 - e^{\lambda_0 p_j})(1 - e^{\theta})\right)  \right)  \\
    & \leq  \exp \left( \frac{n^2}{2} \sum_{j \notin \cW} p_j^2  (e^{-2\theta} - 1) + \theta \left(\tau_n + \frac{n^2}{m} \right) + \frac{32\theta n^3}{m^2} \left( \frac{1 + w}{D_*}\right)^3 \right) \times \\
    & \ \ \ \ \ \ \ \ \ \ \exp \left( n \left\{   \frac{(1+w)}{D_*} \left(1 + \sum_{j \in \cW} p_j (1 - e^{\lambda_0 p_j})(e^{-\theta} - 1) \right)  - 1  - \log\left( \frac{1 + w}{D_*} \right) \right\}  \right).
\end{align*}

We consider the second term on RHS separately, denoting it by $J$. 
$$ J = \frac{(1+w)}{1 + \sum_{j \in \cW} p_j (D_j - 1) } \left(1 + \sum_{j \in \cW} p_j (1 - e^{\lambda_0 p_j})(e^{-\theta} - 1) \right)  - 1  - \log\left( \frac{1 + w}{1 + \sum_{j \in \cW} p_j (D_j - 1)} \right). $$

To analyse the relation between the terms $\sum_{j \in \cW} p_j (D_j - 1)$ and $\sum_{j \in \cW} p_j (1 - e^{\lambda_0 p_j})(e^{-\theta} - 1)$, we define two functions $D(x)$ and $E(x)$ as follows:
\begin{align*}
    D(x) &:= \frac{e^{-\theta} + e^{-x}(1 - e^{-\theta})}{1 + xe^{-x}(e^{\theta} - 1)} \\
    E(x) &:= (e^{-\theta} - 1)(1 - e^{-x}).
\end{align*}
Then we have, $D_j = D(\lambda_0 p_j)$ and $E(\lambda_0 p_j) = (1 - e^{\lambda_0 p_j})(e^{-\theta} - 1)$. It is not difficult to note that the function $F(x) = \frac{D(x) - 1}{E(x)}$ is decreasing for all $x \geq 0$ and it satisfies the relation $1 \leq F(x) \leq 1 + e^{\theta}$. Consequently, we have, for all $x \geq 0$, $(1 + e^{\theta})E(x) \leq D(x) - 1 \leq E(x)$ since $E(x) \leq 0$. Thus, we have,
\begin{align*}
    (1 + e^{\theta})\sum_{j \in \cW} p_j (1 - e^{\lambda_0 p_j})(e^{-\theta} - 1) \leq \sum_{j \in \cW} p_j (D_j - 1) \leq \sum_{j \in \cW} p_j (1 - e^{\lambda_0 p_j})(e^{-\theta} - 1).
\end{align*}
If we let $x = \sum_{j \in \cW} p_j (D_j - 1)$, then we can write $J$ as
$$ J(x)= \frac{(1 + w)(1 + \rho x)}{1 + x} - 1 - \log \left(  \frac{(1 + w)}{1 + x} \right),$$
where $\rho \in [(1 + e^{\theta})^{-1}, 1]$.  \\

Since $D(y) \geq e^{-2\theta}$ for all $y \geq 0$, the domain of $x$ is given as $x \in [-\beta(p)(1 - e^{-2\theta}), 0]$. As $\beta(p) < 1$ and $\theta < 0.4$, the function $J(x)$ is increasing throughout the domain of $x$. Consequently, $J(x) \leq J(0) \leq w^2$. Furthermore, over this domain, we can upper bound $J(x)$ as $J(x) = w^2 + 0.2x$.  \\

Lastly, since $D(x)$ is a decreasing function $D_j \leq D\left(\frac{8\lambda_0}{m} \right)$. Once again, using a local linear approximation of $D(x)$, we have the upper bound $D(x) \leq 1 + 0.7 (e^{-\theta} - e^{\theta}) x$. On combining everything, we obtain the following relation.
\begin{align*}
    J & = \frac{(1+w)}{1 + \sum_{j \in \cW} p_j (D_j - 1) } \left(1 + \sum_{j \in \cW} p_j (1 - e^{\lambda_0 p_j})(e^{-\theta} - 1) \right)  - 1  - \log\left( \frac{1 + w}{1 + \sum_{j \in \cW} p_j (D_j - 1)} \right) \\
    & \leq w^2 + 0.2 \sum_{j \in \cW} p_j (D_j - 1) \\
    & \leq w^2 + 0.2 \sum_{j \in \cW} p_j \left(D\left(\frac{8\lambda_0}{m} \right) - 1 \right) \\
    & \leq w^2 + 0.14 \sum_{j \in \cW} p_j \frac{8\lambda_0}{m} (e^{-\theta} - e^{\theta})  \\
    & \leq w^2 + 1.12 \beta(p) \frac{n (1 + w)}{m D_*} (e^{-2\theta} - 1) .
\end{align*}

On plugging this back into the bound for $T_1'$, we obtain,
\begin{align*}
    T_1' & \leq  \exp \left( \frac{n^2}{2} \sum_{j \notin \cW} p_j^2  (e^{-2\theta} - 1) + \theta \left(\tau_n + \frac{n^2}{m} \right) + \frac{32\theta n^3}{m^2} \left( \frac{1 + w}{D_*}\right)^3 \right) \times \\
    & \ \ \ \ \ \ \ \ \ \ \ \ \ \ \ \ \ \ \ \ \   \exp \left(n \left( w^2 + 1.12 \beta(p) \frac{n (1 + w)}{m D_*} (e^{-2\theta} - 1) \right) + nw^2 \right) \\
    & \leq  \exp \left( \frac{n^2}{2m} \left(\sum_{j \notin \cW} p_j^2  + 2 \beta(p) \right)  (e^{-2\theta} - 1) + \theta \left(\tau_n + \frac{n^2}{m} \right)  + \frac{32\theta n^3 }{m^2} \left( \frac{1 + w}{D_*}\right)^3 + nw^2  \right).
\end{align*}

Let $p_j = \frac{1}{m} + \Delta_j$. So $\sum_{j = 1}^m \Delta_j = 0$ and $\sum_{j = 1}^m |\Delta_j| = \gamma$, the actual $\ell_1$ distance to the uniform distribution. Using this, the first term can be written as,
\begin{align*}
     \frac{n^2}{2m} \left(m \sum_{j \notin \cW}p_j^2 + 2 \sum_{j \in \cW}p_j  \right) & \geq \frac{n^2}{2m} \left(m \sum_{j \notin \cW} \left( \frac{1}{m} + \Delta_j \right)^2 + 2 \sum_{j \notin \cW}\left( \frac{1}{m} + \Delta_j \right)  \right) \\
    & \geq \frac{n^2}{2m} \left(1 + m \sum_{j \notin \cW}  \Delta_j^2 +  \frac{|\cW|}{m}  \right) \\
    & \geq \frac{n^2}{2m} \left( 1 + \frac{\gamma^2}{4} \right),
\end{align*}
where the last step follows from $\sum_{j \notin \cW}  |\Delta_j| \geq \gamma/2$. The bound on $n/m$ and~\eqref{eqn:lambda_0_relations} gives us the following relation
\begin{align*}
    \frac{32 \theta n^3}{m^2} \left( \frac{1 + w}{D_*}\right)^3 & \leq \frac{\theta n^2 \epsilon^2 }{8m}.
\end{align*}
Lastly using~\eqref{eqn: w_upper_bound} and~\eqref{eqn:lambda_0_relations}, we have, $\displaystyle nw^2  \leq  \frac{n^3}{{D}_*^4} (1 + w)^4 \left(\sum_{j \notin \cW}   p_j^2 \right)^2   ( 1- e^{-2\theta})^2 \leq   \theta^2 \cdot  \frac{7n^2 \epsilon^2}{4m}  $. On plugging on these bounds in the bound for $T_1'$, we obtain,
\begin{align*}
    T_1' & \leq  \exp \left( -\frac{n^2}{2m} \left( 1 + \frac{\gamma^2}{4} \right)  (2\theta - 4\theta^2) + \theta \left(\tau_n + \frac{n^2}{m} \right)  + \frac{\theta n^2 \epsilon^2 }{8m} +  \frac{7 \theta^2 n^2 \epsilon^2}{4m} \right) \\
    & \leq  \exp \left( \frac{2n^2\theta^2}{m} \left( 1 + \frac{\gamma^2}{4} +  \frac{7 \epsilon^2}{8} \right)   - \theta \left(\frac{ n^2}{m} \left(\frac{\gamma^2}{4} - \frac{\epsilon^2}{8} \right) -\tau_n  \right)   \right) \\
    & \leq  \exp \left( \frac{2n^2\theta^2}{m} \left( 1 + \frac{\gamma^2}{4} +  \frac{7 \epsilon^2}{8} \right)   - \theta \left( \frac{n^2\gamma^2}{8m}  -\tau_n  \right)   \right)
\end{align*}

Note that the particular choice of $\tau$ and corresponding decision instant as used in Alg.~\ref{alg:SCT} satisfy $\frac{n^2\gamma^2}{8m} \geq 2\tau_n$ for all $k \geq 112/\gamma^2$. Thus, if we define $k_0(\gamma) = 112/\gamma^2$, then for all $k \geq k_0(\gamma)$, we have
\begin{align*}
    T_1' & \leq  \exp \left( \frac{4.25 n^2\theta^2}{m}   - \theta \tau_n   \right).
\end{align*}
On plugging $\displaystyle \theta = \frac{m\tau_n}{9.5n^2}$, we obtain, 
\begin{align*}
    T_1' & \leq  \exp \left( -\frac{m\tau_n^2}{18n^2}   \right).
\end{align*}

On evaluating the above expression at decision instant for epoch $k \geq k_0(\gamma)$, that is, at $n = n_k$ and $\tau = \tau_k$, we obtain
\begin{align*}
    T_1' & \leq \exp \left(  -2.5 \log(k + 2/{\delta}) \right) \\
    & \leq \frac{1}{(k+2/\delta)^{2.5}}.
\end{align*}

Similar to the previous case, we us consider the second term of the expression, denoted by $T_2'$ with the above value of $\theta$. Once again, we evaluate it at a decision instant for $k \geq k(\gamma)$. We have,
\begin{align*}
    T_2 & = \left(\frac{D_*}{1 + w}  \right)^n \exp\left( -0.3n + \theta \left(\tau_n + \frac{n^2}{m} \right)\right) \\
    & \leq \exp\left( -0.3n  + \frac{17\theta n^2}{16m}  - n \log((1 + w)/D^*) \right) \\
    & \leq \exp\left( -0.25n \right),
\end{align*}
where the last expression follows using the value of $\theta$ and bound on $w$ obtained from~\eqref{eqn: w_upper_bound}. \\

On combining both the expressions, we obtain the following result for the probability of error at the decision instant corresponding to epoch $k$.
\begin{align*}
    P_e(n_k) & \leq \frac{n_k! e^{n_k}}{2\pi {n_k}^{n_k}} \left( \frac{1}{(k+2/\delta)^{2.5}} \cdot \sqrt{\frac{18\pi}{n_k}} +  \pi e^{-0.25n_k} \right)  \\
    & \leq e^{1/12n_k}   \left( \frac{3}{(k+2/\delta)^{2.5}}  +  \sqrt{\frac{n_k\pi}{2}} e^{-0.25n_k} \right),
\end{align*}
where the last line again employs Stirling's Approximation. \\

Using the above expression for $P_e(n_k)$, we can obtain the probability for miss detection, which we denote by $\Pr(\text{err}|H_1)$. 
\begin{align*}
	\Pr(\text{err}|H_1) & = \Pr \left( \bigcap_{k = 1}^{\kappa} \left\{ \E_u[K_1(n_k)] - K_1(n_k) < \tau_{k} \right\} \right) \\
	& \leq  \Pr \left(  K_1^{(n_{\kappa})} - \E_u[K_1^{(n_{\kappa})}] > -\tau_{\kappa} \right) \\
	& \leq  P_e(n_{\kappa}) \\
	& \leq e^{1/12n_{\kappa}}   \left( \frac{3}{(\kappa +2/\delta)^{2.5}}  +  \sqrt{\frac{n_{\kappa}\pi}{2}} e^{-0.25n_{\kappa}} \right) \\ 
	& \leq \delta.
\end{align*}

Thus, we have shown that probability of miss detection is also upper bounded by $\delta$. For the last part, we established the expected sample complexity of the routine. Let $\Gamma$ denote the random number of sample taken by the procedure. For the scenario when the underlying distribution is uniform, we use a simple upper bound given as
\begin{align*}
    \E[\Gamma|H_0] & \leq n_{\kappa} \\
    & \leq \frac{112 \sqrt{m}}{\varepsilon^2} \sqrt{\log \left( \frac{112}{\varepsilon^2} + \frac{2}{\delta} \right)} + 1.
\end{align*}.

We now consider the case when the underlying distribution belongs to $\cC(\varepsilon)$ such that $\|p - u\|_1 = \gamma > \varepsilon$. The expected sample complexity is given as 
\begin{align*}
    \E[\Gamma|H_1] & = \sum_{k = 1}^{\kappa} n_{k} \Pr(\Gamma = n_k) \\
    & \leq n_{k_0(\gamma)} +   \sum_{k = k_0(\gamma) + 1}^{\kappa} n_{k} \Pr(\Gamma = n_k) \\
    & \leq n_{k_0(\gamma)} +   \sum_{k = k_0(\gamma) }^{\kappa} n_{k+1} \Pr(\Gamma > n_{k}) \\
    & \leq n_{k_0(\gamma)} +   \sum_{k = k_0(\gamma) }^{\kappa} n_{k+1} P_e(n_k) \\
    & \leq n_{k_0(\gamma)} +   \sum_{k = k_0(\gamma) }^{\kappa} n_{k+1}e^{1/12n_k}   \left( \frac{3}{(k+2/\delta)^{2.5}}  +  \sqrt{\frac{n_k\pi}{2}} e^{-0.25n_k} \right) \\
    & \leq n_{k_0(\gamma)} +   e^{1/12} \sum_{k = k_0(\gamma) }^{\kappa}   \left( \frac{4.5 \sqrt{m \log(k + 2/\delta)}}{(k+2/\delta)^{1.5}}  +  \sqrt{\frac{9\pi}{8}} n_k^{1.5} e^{-0.25n_k} \right) \\
    & \leq n_{k_0(\gamma)} +   e^{1/12}    \left( 4.5 \sqrt{m }\int_{k_0(\gamma) -1}^{\infty} \frac{ \log(x + 2/\delta)}{(x+2/\delta)^{1.5}} dx  +  \sqrt{\frac{9\pi}{8}} \sum_{k = k_0(\gamma) }^{\kappa} n_k^{1.5} e^{-0.25n_k} \right) \\
    & \leq n_{k_0(\gamma)} +   e^{1/12}    \left( 27 \sqrt{m} \left(k_0(\gamma) -1 + \frac{2}{\delta}\right)^{-1/2} \log\left(k_0(\gamma) -1 + \frac{2}{\delta}\right)  +  \sqrt{\frac{9\pi}{8}} C_0 \delta \sqrt{m}  \right),
\end{align*}
for some universal constant $C_0$. The dominant term in the above expression is $n_{k_0(\gamma)}$, giving us, $\displaystyle \E[\Gamma|H_1] = O\left( \gamma^{-2}\sqrt{m \log\left( \frac{1}{\gamma} + \frac{1}{\delta} \right)} \right)$, as required.

\section{Proof of Theorem~\ref{thm:abc_samp_complexity}}
\label{proof:theorem_abc_samp_complexity}

The proof of this theorem is heavily built on the proof of the previous theorem. The basic idea is to first show that for a fine enough discretization, the $\ell_1$ distance of resulting discrete distribution is the same as that of the continuous distribution upto a constant factor. Once this relation is established, we can simply invoke the results obtained in the previous theorem to obtain the results. We begin with the following lemma that relates the $\ell_1$ distances of the discrete and continuous distributions. 

\begin{lemma}
Let $p$ be a distribution in $\cC_L(\varepsilon)$ such that $\| p - u\|_1 = \gamma$ and let $p^{\Delta}$ be the discrete distribution obtained by a uniform discretization of the interval $[0,1]$ into $m$ bins. The $\ell_1$ distance of $p^{\Delta}$ from the uniform distribution, denoted by $[\gamma]_m$ satisfies the following relation
\begin{align*}
    [\gamma]_m = \sum_{i = 1}^m |p^{\Delta}_i - 1/m| \geq \gamma - L/m,
\end{align*}
where $p^{\Delta}_i$ denotes the probability mass of $p^{\Delta}$ in the $i^{\text{th}}$ bin. \label{lemma:l1_distance}
\end{lemma}
\begin{proof}
The proof of the result uses the Lipschitz continuity of the PDF of $p$ to bound the error between the $\ell_1$ distances of the continuous and the discrete distributions. We denote the continuous uniform distribution on $[0, 1]$ using $u(x)$. From the definition of $\ell_1$ distance, we have,
\begin{align*}
	\gamma &= \int_0^1 |p(x) - u(x)| dx \\
	& = \sum_{i = 0}^{m-1} \int_{i/m}^{(i+1)/m} |p(x) - u(x)| dx \\
	& = \sum_{i = 0}^{m-1} \int_{i/m}^{(i+1)/m} \left|p(x) - mp_i^{\Delta} + mp_i^{\Delta} - u(x)\right| dx \\
	& \leq \sum_{i = 0}^{m-1} \int_{i/m}^{(i+1)/m} (|p(x) - mp_i^{\Delta}| + |mp_i^{\Delta} - 1|) dx \\
	& \leq \sum_{i = 0}^{m-1} \int_{i/m}^{(i+1)/m} \left(\frac{L}{m}  + |mp_i^{\Delta} - 1|\right) dx \\
	& \leq \sum_{i = 0}^{m-1} \left[\frac{L}{m^2}  + \left|\int_{(i-1)/m}^{i/m} p(y) dy - \frac{1}{m}\right| \left(\int_{i/m}^{(i+1)/m} m dx \right) \right] \\
	& \leq \sum_{i = 0}^{m-1} \frac{L}{m^2}  + \sum_{i = 0}^{m-1} \left|p_i^{\Delta} - \frac{1}{m}\right|   \\
	& \leq \frac{L}{m} + [\gamma]_{m}.
\end{align*}
In the fifth step, we have used the mean value theorem along with the Lipschitz condition on the PDF. From the mean value theorem, we can conclude that there exists an $x_i \in [i/m, (i+1)/m]$ such that $p(x_i) = p_i^{\Delta}/(1/m) = m p_i$ for all $i = 0,1,2, \dots, m-1$ and since PDF is $L$-Lipschitz, we have $|p(x) - p(x_i)| \leq L/m$ for all $x \in [i/m, (i+1)/m]$. Consequently, we can relate the $\ell_1$ distances between the continuous and discrete distributions. In particular, if $m \geq L/2\gamma$, then $[\gamma]_{m} \geq \gamma/2$.
\end{proof}

Once we have obtained a discretization, the proof for the probability of error and sample complexity is almost the same as that of previous case with very minor modifications. Firstly, we can simply modify the set $\cW$ to $\cW_k$ defined for each epoch as $\cW_k = \{j : p^{\Delta}_j \geq 8/m_k \}$, where once again $p^{\Delta}_j$ is the mass in the $j^{\text{th}}$ bin in the discretization. Also for this case, instead of computing the error for any sample $n$, we just compute it for the decision instant $n_k$ and the corresponding number of bins $m_k$. \\

Once again we begin with the probability of false alarm. It can be verified that for the choice of $n_k$ and $m_k$, all the conditions in previous analysis are satisfied yielding us same bounds on $T_1$ and $T_2$ and consequently, the following result holds for the probability of error at the decision instant corresponding to epoch $k$.
\begin{align*}
    P_e(n_k)  & \leq e^{1/12n_k}   \left( \frac{3}{(k+2/\delta)^3}  +  \sqrt{\frac{n_k\pi}{2}} e^{-0.25n_k} \right).
\end{align*}
Using a similar process as in the previous proof, we can obtain the probability for false alarm, which we denote by $\Pr(\text{err}|H_0)$. 
\begin{align*}
	\Pr(\text{err}|H_0) & = \Pr \left( \bigcup_{k = 1}^{\kappa} \left\{ \E_u[K_1(n_k)] - K_1(n_k) > \tau_{k} \right\} \right) \\
	& \leq \sum_{k = 1}^{\kappa} \Pr \left(  K_1^{(n_k)} - \E_u[K_1^{(n_k)}] < -\tau_k \right) \\
	& \leq \sum_{k = 1}^{\kappa} P_e(n_k) \\
	& \leq \sum_{k = 1}^{\kappa} e^{1/12n_k}   \left( \frac{3}{(k+2/\delta)^3}  +  \sqrt{\frac{n_k\pi}{2}} e^{-0.25n_k} \right) \\
	& \leq \delta.
\end{align*}
The last line follows by the same reasoning as in the proof of the previous theorem. \\

We consider the probability of missed detection in a similar manner. In this scenario, the $\ell_1$ distance to the uniform distribution of the underlying distribution $p$ is $\gamma$ and for simplicity, we denote the discretized $\ell_1$ distance as $[\gamma]_k$ instead of $[\gamma]_{m_k}$. If we define $k_0(\gamma) = \min\{k : k \geq 144 [\gamma]_k^{-2}\}$, then using this definition of $k_0(\gamma)$, we can obtain all the results from the previous theorem. Specifically, for $k \geq k_0(\gamma)$, we have,
\begin{align*}
    \frac{32 \theta n_k^3}{m_k^2} \left( \frac{1 + w}{D_*}\right)^3 & \leq \frac{\theta n_k^2 [\gamma]_k^2 }{8m_k}, \\
    n_k w^2  & \leq   \theta^2 \cdot  \frac{7n_k^2 [\gamma]_k^2}{4m_k}, \\ 
     2\tau_{k} & \leq \frac{n_k^2[\gamma]_k^2}{8m_k}.
\end{align*}
Consequently, for $k \geq k_0(\gamma)$, we have, 
\begin{align*}
    T_1' & \leq  \exp \left( \frac{4.25 n_k^2\theta^2}{m_k}   - \theta \tau_k   \right).
\end{align*}
On plugging $\displaystyle \theta = \frac{m_k\tau_k}{9.5n_k^2}$, we obtain, 
\begin{align*}
    T_1' & \leq  \exp \left( -\frac{m_k\tau_k^2}{18n_k^2}   \right) \\
    & \leq \frac{1}{(k + 2/\delta)^{4.5}}.
\end{align*}
It is not difficult to see that $T_2' \leq \exp(-0.25n_k)$ for $k \geq k_0(\gamma)$, yielding a similar expression for $P_e(n_k)$ and consequently we can conclude that $\Pr(\text{err}|H_1) \leq \delta$. \\

The expected sample complexity for the uniform case is simply bounded as $n_{\kappa}$ implying that $\E[\Gamma | H_0] $ is $O(\varepsilon^{-6} \log(\varepsilon^{-1} + \delta^{-1})$. Lastly, the sample complexity for case when the underlying distribution belongs to $\cC_L(\varepsilon)$ such that $\|p - u\|_1 = \gamma > \varepsilon$ is given as follows.
\begin{align*}
    \E[\Gamma|H_1] & = \sum_{k = 1}^{\kappa} n_{k} \Pr(\Gamma = n_k) \\
    & \leq n_{k_0(\gamma)} +   \sum_{k = k_0(\gamma) + 1}^{\kappa} n_{k} \Pr(\Gamma = n_k) \\
    & \leq n_{k_0(\gamma)} +   \sum_{k = k_0(\gamma) }^{\kappa} n_{k+1} \Pr(\Gamma > n_{k}) \\
    & \leq n_{k_0(\gamma)} +   \sum_{k = k_0(\gamma) }^{\kappa} n_{k+1} P_e(n_k) \\
    & \leq n_{k_0(\gamma)} +   \sum_{k = k_0(\gamma) }^{\kappa} n_{k+1}e^{1/12n_k}   \left( \frac{3}{(k+2/\delta)^{4.5}}  +  \sqrt{\frac{n_k\pi}{2}} e^{-0.25n_k} \right) \\
    & \leq n_{k_0(\gamma)} +   e^{1/12} \sum_{k = k_0(\gamma) }^{\kappa}   \left( \frac{3 \sqrt{c_{0}} \log(k + 2/\delta)}{(k+2/\delta)^{1.5}}  +  7.5\sqrt{\frac{\pi}{2}} n_k^{1.5} e^{-0.25n_k} \right) \\
    & \leq n_{k_0(\gamma)} +   e^{1/12}    \left(3 \sqrt{c_{0}}\int_{k_0(\gamma) -1}^{\infty} \frac{ \log(x + 2/\delta)}{(x+2/\delta)^{1.5}} dx  +  7.5\sqrt{\frac{\pi}{2}} \sum_{k = k_0(\gamma) }^{\kappa} n_k^{1.5} e^{-0.25n_k} \right) \\
    & \leq n_{k_0(\gamma)} +   e^{1/12}    \left( 18 \sqrt{c_{0}} \left(k_0(\gamma) -1 + \frac{2}{\delta}\right)^{-1/2} \log\left(k_0(\gamma) -1 + \frac{2}{\delta}\right)  +  7.5\sqrt{\frac{\pi}{2}} C_1 \delta \sqrt{c_{0}}  \right),
\end{align*}
for some universal constant $C_1$. The dominant term in the above expression is $n_{k_0(\gamma)}$. From the particular choice of $m_k$ and Lemma~\ref{lemma:l1_distance}, we can conclude that $k_0(\gamma) \leq 576\gamma^{-2}$. This bound along with the expression for $n_{k_0(\gamma)}$ yields $\displaystyle \E[\Gamma|H_1] = O\left( \gamma^{-6} \log\left( {\gamma}^{-1} + {\delta}^{-1} \right) \right)$, as required.

\section{Proof of Theorem~\ref{thm:lower_bound}}
\label{proof:theorem_lower_bound}

In this section, we provide the proof of the lower bound on the number of samples required for the case of continuous distributions. The basic idea is very similar to the discrete case and involves constructing a mixture of distributions and then invoking LeCam's two point method to obtain a lower bound. Consider a $\Delta > 0$. The particular value of $\Delta$ will be determined later in the analysis. We define the following two functions on the interval $[0, \Delta]$.
\begin{align*}
    f_1(x) & = \begin{cases} 1 + Lx & \text{ if } 0 \leq x \leq \Delta/2 \\ 1 + L(\Delta - x) & \text{ if } \Delta/2 \leq x \leq \Delta  \end{cases} \\
    f_{-1}(x) & = \begin{cases} 1 - Lx & \text{ if } 0 \leq x \leq \Delta/2 \\ 1 - L(\Delta - x) & \text{ if } \Delta/2 \leq x \leq \Delta  \end{cases}.
\end{align*}
The value of $\Delta$ is small enough to ensure both these functions are non-negative. It is evident from the construction that both $f_1$ and $f_{-1}$ are $L$-Lipschitz functions. Additionally note that
\begin{align*}
    \int_0^{\Delta} |f_1(x) - 1| \ dx & = \int_0^{\Delta} |f_{-1}(x) - 1| \ dx  \\
    & = \int_0^{\Delta/2} Lx \ dx +  \int_{\Delta/2}^{\Delta} L(\Delta - x) \ dx  \\
    & = 2\int_0^{\Delta/2} Lx \ dx  \\
    & = \frac{L \Delta^2}{4}.
\end{align*}

Let $M = \Delta^{-1}$, for simplicity is assumed to be an even integer and $\cZ$ denote the set of $2^{M/2}$ binary strings of the form $\{ \pm 1 \}^{M/2}$. We define a class of distributions, $\cQ$, containing $2^{M/2}$ distributions and is defined using a bijection from the set $\cZ$. In particular, the density $f_q$ distribution $q \in \cQ$ corresponding to $z = (z_0, z_1, \dots, z_{M/2-1}) \in \cZ$ is defined as a piecewise function for $i = 0,1,2, \dots, M/2-1$ as follows
\begin{align*}
    f_q(x) = \begin{cases} f_{z_i}(x) & \text{ for } x \in [2i\Delta, (2i+1)\Delta] \\ f_{-z_i}(x) & \text{ for } x \in [(2i +1)\Delta, 2(i+1)\Delta] \end{cases}
\end{align*}
The $\ell_1$ distance of any distribution $q \in \cQ$ from the uniform distribution is given as
\begin{align*}
    \int_{0}^1 |f_q(x) - 1| \ dx & = \sum_{i = 0}^{M/2-1} \int_{2i\Delta}^{(2i+1)\Delta} |f_{z_i} (x) - 1| \ dx +  \int_{(2i+1)\Delta}^{2(i+1)\Delta} |f_{-z_i} (x) - 1| \ dx  \\
    & = \sum_{i = 0}^{M-1} \frac{L \Delta^2}{4} \\
    & = \frac{L \Delta}{4}.
\end{align*}
Thus, the $\ell_1$ distance to uniform distribution is the same for all the distributions in $\cQ$, equal to $L \Delta/4$. \\

Suppose there exists an algorithm such that using $n$ i.i.d. samples, it can distinguish between the uniform distribution and a distribution chosen uniformly at random from $\cQ$ with probability at least $9/10$ . Then, from LeCam's two point method~\citep{LeCam1986} we know that $\| \nu_{\pi_1}^n - \nu_{\pi_2}^n \|_{TV} \geq 0.8$ and consequently $\chi^2(\nu_{\pi_1}^n, \nu_{\pi_2}^n) > 1.28 $. Here $\| p - q \|_{TV}$ is the total variation distance between the distributions $p$ and $q$ and $\chi^2(p, q)$ is the $\chi^2$-distance between those distribution. Lastly, $\nu_{\pi_1}^n$ and $\nu_{\pi_2}^n$ are distributions given as follows. $\nu_{\pi_1}^n$ is the joint distribution of $n$ samples generated i.i.d. from a uniform distribution. On the other hand, $\nu_{\pi_2}^n$ is the mixture distribution obtained by first choosing a distribution from $\cQ$ uniformly at random and then taking $n$ i.i.d. samples from it. Mathematically, $\nu_{\pi_2}^n$ is given as 
\begin{align*}
    \nu_{\pi_2}^n = \frac{1}{2^{M/2}} \sum_{q \in \cQ} q(x^n).
\end{align*}

To simplify the analysis, we use the standard Poissonization technique wherein instead of obtaining $n$ samples directly, we first sample $N$ from a Poisson distribution with mean $n$ and then $N$ samples from the underlying distribution. Our objective is to obtain an upper bounds on the $\chi^2$-distance between $\nu_{\pi_1}^N$ and $\nu_{\pi_2}^N$ where the $N$ is used in the notation to denote the Poissonization sampling technique and taking the expectation over the Poisson random variable. Using the properties of the Poisson random variable, we have the following relation: $\chi^2(\nu_{\pi_1}^N, \nu_{\pi_2}^N) \geq 0.4 \chi^2(\nu_{\pi_1}^n, \nu_{\pi_2}^n) $. Thus, it is sufficient to bound $\chi^2(\nu_{\pi_1}^N, \nu_{\pi_2}^N)$ as that would directly give us an upper bound on $\chi^2(\nu_{\pi_1}^n, \nu_{\pi_2}^n)$. \\

Let us focus on the distribution of points under $\nu_{\pi_2}$ in the first two intervals, $[0, \Delta]$ and $[\Delta, 2\Delta]$. The density in both these interval is just determined by value of $z_0$. Let $\mathbf{x}^0$ denote the set of samples in $[0, \Delta]$ and $\mathbf{y}^0$ denote the set of samples in $[\Delta, 2\Delta]$. We denote the size of any such set $\mathbf{x}$ as $|\mathbf{x}|$. We first evaluate the conditional density when $z_0 = 1$. We have,
\begin{align*}
    & \nu_{\pi_2}^N \left( \mathbf{x} = (x_1^0, x_2^0, \dots, x_{n_1}^0), \mathbf{y} = (y_1^0, y_2^0, \dots, y_{n_2}^0), |\mathbf{x}^0| = n_1, |\mathbf{y}^0| = n_2  \ \bigg| \ z_0 = 1\right) \\
    & = \frac{e^{-n p_{+1}}}{n_1!}(n p_{+1})^{n_1} \left(\prod_{i = 1}^{n_1}\frac{f_{1}(x_i^0)}{p_{+1}} \right) \frac{e^{-n p_{-1}}}{n_2!}(n p_{-1})^{n_2} \left(\prod_{i = 1}^{n_2}\frac{f_{-1}(y_i^0)}{p_{-1}} \right) \\
    & = \frac{e^{-n (p_{+1} + p_{-1})}}{n_1! n_2!}n^{n_1 + n_2} \left(\prod_{i = 1}^{n_1} f_{1}(x_i^0) \right)  \left(\prod_{i = 1}^{n_2}f_{-1}(y_i^0) \right).
\end{align*}
In the above expression, the terms $p_{+1}$ and $p_{-1}$ are given by
\begin{align*}
    p_{+1} & = \int_0^{\Delta} f_1(x) \ dx = \Delta + \frac{L\Delta^2}{2} \\
    p_{-1} & = \int_0^{\Delta} f_{-1}(x) \ dx = \Delta - \frac{L\Delta^2}{2} \\
\end{align*}
We can similarly evaluate the conditional density for the case when $z_0 = -1$ and combine it with the above result to obtain the density for $\nu_{\pi_2}^N$ as follows.
\begin{align*}
    & \nu_{\pi_2}^N \left( \mathbf{x}^0 = (x_1^0, x_2^0, \dots, x_{n_1}^0), \mathbf{y}^0 = (y_1^0, y_2^0, \dots, y_{n_2}^0), |\mathbf{x}^0| = n_1, |\mathbf{y}^0| = n_2 \right) \\
    & = \frac{e^{-2n\Delta}}{2n_1! n_2!}n^{n_1 + n_2}  \left[\left(\prod_{i = 1}^{n_1} f_{1}(x_i^0) \right)  \left(\prod_{i = 1}^{n_2}f_{-1}(y_i^0) \right) + \left(\prod_{i = 1}^{n_1} f_{-1}(x_i^0) \right)  \left(\prod_{i = 1}^{n_2}f_{1}(y_i^0) \right) \right] .
\end{align*}
Since $\nu_{\pi_1}^N$ is effectively the uniform distribution, the corresponding expression for $\nu_{pi_1}^N$ is given by
\begin{align*}
    & \nu_{\pi_1}^N \left( \mathbf{x}^0 = (x_1^0, x_2^0, \dots, x_{n_1}^0), \mathbf{y}^0 = (y_1^0, y_2^0, \dots, y_{n_2}^0), |\mathbf{x}^0| = n_1, |\mathbf{y}^0| = n_2 \right) \\
    & = \frac{e^{-2n\Delta}}{n_1! n_2!}n^{n_1 + n_2}.
\end{align*}

We can extend this expression to any interval. In particular, the set of samples obtained in the interval $[2i \Delta, (2i + 1)\Delta]$ is denoted by $\mathbf{x}^i$ with $|\mathbf{x}^i| = n_{2i+1}$ and similarly the set of samples obtained in the interval $[(2i + 1) \Delta, 2(i + 1)\Delta]$ is denoted by $\mathbf{y}^i$ with $|\mathbf{y}^i| = n_{2i+2}$ for $i = 0,1, \dots M/2 -1$. Then for any such collection of $\mathbf{x}^i$'s and $\mathbf{y}^i$'s, we have, 
\begin{align*}
    & \nu_{\pi_2}^N \left( \mathbf{x}^0, \mathbf{x}^1, \dots, \mathbf{x}^{M/2 -1}, \mathbf{y}^0, \mathbf{y}^1, \dots, \mathbf{y}^{M/2 -1}  \right) \\
    & = \prod_{j = 0}^{M/2 - 1} \left\{ \frac{e^{-2n\Delta}}{2n_{2j+1}! n_{2j+2}!}n^{n_{2j+1} + n_{2j+2}}  \left[\left(\prod_{i = 1}^{n_{2j+1}} f_{1}(x_i^j) \right)  \left(\prod_{i = 1}^{n_{2j+2}}f_{-1}(y_i^j) \right) + \left(\prod_{i = 1}^{n_{2j+1}} f_{-1}(x_i^j) \right)  \left(\prod_{i = 1}^{n_{2j+2}}f_{1}(y_i^j) \right) \right] \right\}.
\end{align*}
Similarly, we have,
\begin{align*}
    & \nu_{\pi_1}^N \left( \mathbf{x}^0, \mathbf{x}^1, \dots, \mathbf{x}^{M/2 -1}, \mathbf{y}^0, \mathbf{y}^1, \dots, \mathbf{y}^{M/2 -1}  \right) \\
    & = \prod_{j = 0}^{M/2 - 1} \left\{ \frac{e^{-2n\Delta}}{n_{2j+1}! n_{2j+2}!}n^{n_{2j+1} + n_{2j+2}} \right\}.
\end{align*}

The $\chi^2$ distance between $\nu_{\pi_1}^N$ and $\nu_{\pi_2}^N$ is given as,
\begin{align*}
    \chi^2(\nu_{\pi_2}^N, \nu_{\pi_1}^N) & = \E_{\nu_{\pi_2}^N} \left[\frac{\nu_{\pi_2}^N}{\nu_{\pi_1}^N} \right]  - 1 \\
    & = \E_{\nu_{\pi_2}^N}\left[ g(\mathbf{x}^0, \dots,  \mathbf{x}^{M/2 -1}, \mathbf{y}^0, \dots, \mathbf{y}^{M/2 -1}) \right]    - 1 ,
\end{align*}
where
\begin{align*}
    g(\mathbf{x}^0, \dots, \mathbf{y}^{M/2 -1}) = \prod_{j = 0}^{M/2 - 1} \frac{1}{2}\left[\left(\prod_{i = 1}^{n_{2j+1}} f_{1}(x_i^j) \right)  \left(\prod_{i = 1}^{n_{2j+2}}f_{-1}(y_i^j) \right) + \left(\prod_{i = 1}^{n_{2j+1}} f_{-1}(x_i^j) \right)  \left(\prod_{i = 1}^{n_{2j+2}}f_{1}(y_i^j) \right) \right].
\end{align*}
The density function depends on vector $z$. Owing to symmetry and independence, we just focus on evaluating one term in the product, in particular $j=0$ (corresponding to $z_0$) as all the terms will be equal in expectation. We first evaluate the conditional expectation with respect to underlying distribution $q$ with $z$ and $N$ fixed, followed by taking the expectation over $z$ and $N$ as well. \\

We begin with evaluating the conditional expectations over $z_0 = 1$ which we denote by $E_1$. We have,
\begin{align*}
    & E_1 = \E \left[\frac{1}{2} \left\{\left(\prod_{i = 1}^{n_1} f_{1}(x_i^0) \right)  \left(\prod_{i = 1}^{n_2}f_{-1}(y_i^0) \right) + \left(\prod_{i = 1}^{n_1} f_{-1}(x_i^0) \right)  \left(\prod_{i = 1}^{n_2}f_{1}(y_i^0) \right) \right\} \bigg| z_0 = 1\right]  \\
    & = \int   \left[\frac{1}{2} \left\{ \left(\prod_{i = 1}^{n_1} f_{1}(x_i^0) \right)  \left(\prod_{i = 1}^{n_2}f_{-1}(y_i^0) \right) + \left(\prod_{i = 1}^{n_1} f_{-1}(x_i^0) \right)  \left(\prod_{i = 1}^{n_2}f_{1}(y_i^0) \right)  \right\} \right] \left(\prod_{i = 1}^{n_1} \frac{f_{1}(x_i^0)}{p_{+1}} \right)  \left(\prod_{i = 1}^{n_2}\frac{f_{-1}(y_i^0)}{p_{-1}} \right) d \mathbf{x}^0 d \mathbf{y}^0.
\end{align*}
The above expression can be simplified using the following integrals.
\begin{align*}
    \int_0^{\Delta} f_1^2(x) dx & = \int_0^{\Delta/2} (1 + Lx)^2 dx + \int_0^{\Delta/2} (1 + L(\Delta - x))^2 dx \\
    & = 2\int_0^{\Delta/2} (1 + Lx)^2 dx  \\
    & = \frac{2}{3L} \left[ \left(1 + \frac{L\Delta}{2} \right)^3  - 1 \right].
\end{align*}
Similarly, 
\begin{align*}
    \int_0^{\Delta} f_{-1}^2(x) dx & = \int_0^{\Delta/2} (1 - Lx)^2 dx + \int_0^{\Delta/2} (1 - L(\Delta - x))^2 dx \\
    & = 2\int_0^{\Delta/2} (1 - Lx)^2 dx  \\
    & = \frac{2}{3L} \left[1 -  \left(1 - \frac{L\Delta}{2} \right)^3 \right].
\end{align*}
Lastly, 
\begin{align*}
    \int_0^{\Delta} f_{1}(x)f_{-1}(x) dx & = \int_0^{\Delta/2} [1 - (Lx)^2] dx + \int_0^{\Delta/2} [1 - L^2(\Delta - x)^2 ] dx \\
    & = 2\int_0^{\Delta/2} [1 - (Lx)^2] dx  \\
    & =\left[\Delta -  \frac{L^2 \Delta^3}{12} \right].
\end{align*}
On plugging these values back into the above expression, we obtain,
\begin{align*}
    & E_1 =  \E \left[ \frac{1}{2}  \left\{\left(\prod_{i = 1}^{n_1} f_{1}(x_i^0) \right)  \left(\prod_{i = 1}^{n_2}f_{-1}(y_i^0) \right) + \left(\prod_{i = 1}^{n_1} f_{-1}(x_i^0) \right)  \left(\prod_{i = 1}^{n_2}f_{1}(y_i^0) \right) \right\} \bigg| z_0 = 1 \right] \\
    & = \frac{1}{2} \bigg[\left\{ \frac{2}{3Lp_{+1}} \left[ \left(1 + \frac{L\Delta}{2} \right)^3  - 1 \right] \right\}^{n_1}\left\{ \frac{2}{3Lp_{-1}} \left[1 -  \left(1 - \frac{L\Delta}{2} \right)^3 \right] \right\}^{n_2} + \\
    & \ \ \ \ \ \ \ \ \ \ \ \ \ \ \ \  \left\{ \frac{1}{p_{+1}} \left[\Delta -  \frac{L^2 \Delta^3}{12} \right] \right\}^{n_1}\left\{ \frac{1}{p_{-1}} \left[\Delta -  \frac{L^2 \Delta^3}{12} \right] \right\}^{n_2} \bigg].
\end{align*}
A similar analysis yields the following relation.
\begin{align*}
    & E_{-1} = \E \left[ \frac{1}{2}  \left\{\left(\prod_{i = 1}^{n_1} f_{1}(x_i^0) \right)  \left(\prod_{i = 1}^{n_2}f_{-1}(y_i^0) \right) + \left(\prod_{i = 1}^{n_1} f_{-1}(x_i^0) \right)  \left(\prod_{i = 1}^{n_2}f_{1}(y_i^0) \right) \right\} \bigg| z_0 = -1 \right] \\
    & = \frac{1}{2} \bigg[\left\{ \frac{2}{3Lp_{+1}} \left[ \left(1 + \frac{L\Delta}{2} \right)^3  - 1 \right] \right\}^{n_2}\left\{ \frac{2}{3Lp_{-1}} \left[1 -  \left(1 - \frac{L\Delta}{2} \right)^3 \right] \right\}^{n_1} + \\
    & \ \ \ \ \ \ \ \ \ \ \ \ \ \ \ \  \left\{ \frac{1}{p_{+1}} \left[\Delta -  \frac{L^2 \Delta^3}{12} \right] \right\}^{n_2}\left\{ \frac{1}{p_{-1}} \left[\Delta -  \frac{L^2 \Delta^3}{12} \right] \right\}^{n_1} \bigg].
\end{align*}

Using these results, we have, 

\begin{align*}
    \chi^2(\nu_{\pi_2}^N, \nu_{\pi_1}^N) & = \E_{\nu_{\pi_2}^N}\left[ g(\mathbf{x}^0, \dots,  \mathbf{x}^{M/2 -1}, \mathbf{y}^0, \dots, \mathbf{y}^{M/2 -1}) \right]    - 1  \\
    & = \left[\frac{1}{2} \left( \E_{z_0 = 1, N} [E_1] + \E_{z_0 = -1, N} [E_{-1}] \right)\right]^{M/2} - 1
\end{align*}

We focus on the first term where $N_1$ and $N_2$ are independent Poisson random variables with means $n p_{+1}$ and $n p_{-1}$ respectively. Using the expression for moment generating function of a Poisson random variable, we have,
\begin{align*}
    & \E_{z_0 =1, N}  \\
    & = \frac{1}{2}\E_{N} \Bigg[ \left\{ \frac{2}{3Lp_{+1}} \left[ \left(1 + \frac{L\Delta}{2} \right)^3  - 1 \right] \right\}^{N_1}\left\{ \frac{2}{3Lp_{-1}} \left[1 -  \left(1 - \frac{L\Delta}{2} \right)^3 \right] \right\}^{N_2} + \\
    & \ \ \ \ \ \ \  \left\{ \frac{1}{p_{+1}} \left[\Delta -  \frac{L^2 \Delta^3}{12} \right] \right\}^{N_1}\left\{ \frac{1}{p_{-1}} \left[\Delta -  \frac{L^2 \Delta^3}{12} \right] \right\}^{N_2} \Bigg] \\
    & = \frac{1}{2}\exp \Bigg( n p_{+1}\left\{ \frac{2}{3Lp_{+1}} \left[ \left(1 + \frac{L\Delta}{2} \right)^3  - 1 \right] - 1 \right\} + n p_{-1}\left\{ \frac{2}{3Lp_{-1}} \left[1 -  \left(1 - \frac{L\Delta}{2} \right)^3 \right] - 1 \right\} \Bigg) + \\
    & \ \ \ \ \ \ \ \frac{1}{2} \exp \left( n p_{+1} \left\{ \frac{1}{p_{+1}} \left[\Delta -  \frac{L^2 \Delta^3}{12} \right]  -1 \right\} + n p_{-1}\left\{ \frac{1}{p_{-1}} \left[\Delta -  \frac{L^2 \Delta^3}{12} \right] - 1\right\} \right) \\
    & = \frac{1}{2} \left[\exp \left( \frac{nL^2 \Delta^3}{6}\right) + \exp \left( -\frac{nL^2 \Delta^3}{6}\right) \right]\\
    & \leq \exp \left( \frac{n^2L^4 \Delta^6}{72}\right).
\end{align*}

Using a very similar analysis, we find that the other term, $\E_{z_0 = -1, N} [E_{-1}]$, can also be upper bounded by the same expression. Thus, we have,

\begin{align*}
    \chi^2(\nu_{\pi_2}^N, \nu_{\pi_1}^N) & = \left[\frac{1}{2} \left( \E_{z_0 = 1, N} [E_1] + \E_{z_0 = -1, N} [E_{-1}] \right)\right]^{M/2} - 1 \\
    & \leq \exp \left( \frac{n^2L^4 \Delta^6 M}{72}\right)  - 1 \\
    & \leq \exp \left( \frac{n^2L^4 \Delta^5}{72}\right)  - 1.
\end{align*}

In order to ensure the lower bound on $\chi^2$ distance required for any viable tester, we need to ensure that this upper bounds is at least as large as the lower bounds. Consequently, using the relation between and $\chi^2(\nu_{\pi_1}^n, \nu_{\pi_2}^n)$ and $\chi^2(\nu_{\pi_1}^N, \nu_{\pi_2}^N)$ along with the above relation, we have, 
\begin{align*}
    \exp \left( \frac{n^2L^4 \Delta^5}{72}\right)  - 1 \geq 0.512 \implies n \geq \frac{c}{L^2 \Delta^{2.5}},
\end{align*}
for some universal constant $c$. If we set $\Delta = 4(1 +\eta)\varepsilon/L$, for $\eta > 0$ then the $\ell_1$ distance of any $q \in \cQ$ would be $(1 + \eta) \varepsilon$ implying $\cQ \subset \cC_L(\varepsilon)$. The lower bound follows by plugging this value of $\Delta$.

\end{document}